\documentclass{msc}

\usepackage[utf8]{inputenc}
\usepackage[T1]{fontenc}

\usepackage{csquotes}

\usepackage[labelfont=bf,textfont=it]{caption}

\usepackage{enumitem}

\usepackage{microtype}

\usepackage{mathtools, thmtools}

\usepackage{float}

\usepackage{mathrsfs}
\usepackage{amssymb, amsfonts}

\usepackage{prftree}

\usepackage{centernot}

\usepackage[thicklines]{cancel}

\usepackage{stackrel}

\usepackage{tikz}

\usetikzlibrary{cd, arrows, decorations.pathreplacing}

\newtheorem{example}[therm]{Example}

\usepackage[toc, page]{appendix}

\usepackage[capitalise]{cleveref}

\usepackage{fonttable}
\usepackage{stmaryrd}
\usepackage{tabularx}
\usepackage{expl3}
\usepackage{xparse, xpatch}
\usepackage{stackengine, old-arrows}

\usepackage{quiver}

\newenvironment{eqalign*}{\begin{equation*}\begin{aligned}}{\end{aligned}\end{equation*}}

\newenvironment{diagram*}{\begin{equation*}\begin{tikzcd}}{\end{tikzcd}\end{equation*}}

\tikzset{
  relation/.style={
    draw=none,
    every to/.append style={
      edge node={node [sloped, allow upside down, auto=false]{$#1$}}}
  }
}

\tikzcdset{
  mono/.code={
    \pgfsetarrows{tikzcd to reversed-tikzcd to}
  }
}
\tikzcdset{
  epi/.code={
    \pgfsetarrowsend{tikzcd double to}
  }
}
\tikzcdset{
  into/.code={
    \pgfsetarrows{tikzcd right hook-tikzcd to}
  }
}
\tikzcdset{
  twocell/.style={Rightarrow, shorten >= 3ex, shorten <= 3ex}
}

\tikzcdset{
  row sep/normal={
    6ex
  },
  column sep/normal={
    8ex
  }
}

\newcommand{\Overset}[2]{%
  \mathop{#2}\limits^{\vbox to -.1ex{%
  \kern -1.8ex\hbox{$#1$}\vss}}%
}
\newcommand{\Underset}[2]{%
  \mathop{#2}\limits_{\vbox to .1ex{%
  \kern -.6ex\hbox{$#1$}\vss}}%
}

\mathchardef\dash="2D

\newcommand{\ot}{\leftarrow}
\newcommand{\from}{\ot}

\newcommand{\longto}{\longrightarrow}

\newcommand{\twoto}{\Rightarrow}

\newcommand{\narrow}[2]{\overset{#1}{#2}}
\newcommand{\nto}[1]{\narrow{#1}{\to}}
\newcommand{\nfrom}[1]{\narrow{#1}{\from}}
\newcommand{\nlongfrom}[1]{\xleftarrow{#1}}
\newcommand{\nlongto}[1]{\xrightarrow{#1}}

\newcommand{\opticto}{\leftrightarrows}
\newcommand{\chartto}{\rightrightarrows}

\newcommand{\R}{\mathbb{R}}

\newcommand{\biglens}[2]{
	 \begin{pmatrix}{\vphantom{f_f^f}#1} \\ {\vphantom{f_f^f}#2} \end{pmatrix}
}
\newcommand{\littlelens}[2]{
	 \begin{psmallmatrix}{\vphantom{f}#1} \\ {\vphantom{f}#2} \end{psmallmatrix}
}
\newcommand{\lens}[2]{
  \relax\if@display
	 \biglens{#1}{#2}
  \else
	 \littlelens{#1}{#2}
  \fi
}

\usepackage{xstring}
\newcommand{\cat}[1]{
  \relax
  \StrLen{#1}[\catarglen]
  \ifnum\catarglen=1
    \mathcal{#1}
  \else
    \mathsf{#1}
  \fi
}
\newcommand{\dblcat}[1]{\cat{\mathbb #1}}

\newcommand{\cod}{\mathsf{cod}}
\newcommand{\dom}{\mathsf{dom}}

\newcommand{\identity}{\mathsf{Id}}
\newcommand{\id}{\identity}

\newcommand{\adj}{\dashv}

\newcommand{\Cat}{\dblcat{Cat}}

\newcommand{\Set}{\cat{Set}}

\newcommand{\CMon}{\cat{CMon}}

\newcommand{\Smooth}{\cat{Smooth}}

\newcommand{\Fib}{\dblcat{Fib}}

\DeclareFontFamily{U}{musix}{}%
\DeclareFontShape{U}{musix}{m}{n}{%
  <-12>   musix11
  <12-15> musix13
  <15-18> musix16
  <18-23> musix20
  <23->   musix29
}{}%
\newcommand*\musix{\usefont{U}{musix}{m}{n}\selectfont}
\DeclareTextFontCommand{\textmusix}{\musix}

\usepackage{xypic}

\newcommand{\simple}{\mathsf{simp}}
\renewcommand{\lens}{\mathsf{lens}}
\newcommand{\dlens}{\mathsf{dLens}}

\newcommand{\bo}{\mathsf{bo}}
\newcommand{\ff}{\mathsf{ff}}

\newcommand{\Cart}{\dblcat{Cart}}

\newcommand{\BundleCat}{{\dblcat{Bun}\Cat}}
\newcommand{\CartMonCat}{{\Cart\Cat}}
\newcommand{\SemiaddCat}{{\dblcat{Smadd}\Cat}}

\newcommand{\CLA}{{\dblcat{CLA}}}
\newcommand{\CLACat}{{\CLA\Cat}}

\newcommand{\lin}{\mathsf{lin}}
\newcommand{\add}{\mathsf{add}}
\newcommand{\nadd}{{\cancel{\add}}}
\newcommand{\nlin}{{\cancel{\lin}}}

\newcommand{\Lin}{\cat{Lin}}

\newcommand{\Diff}{\cat{Diff}}

\newcommand{\fst}{{\mathsf{fst}}}

\newcommand{\<}{\langle}
\renewcommand{\>}{\rangle}

\newcommand{\FOTS}{\begin{tikzcd}[ampersand replacement=\&,cramped]
    {\cat B^{\downbundle+}} \& {\cat B}
    \arrow["\cod"{description}, from=1-1, to=1-2]
    \arrow["T"{description}, curve={height=12pt}, from=1-2, to=1-1]
\end{tikzcd}}

\newcommand{\bundle}{\twoheadrightarrow}
\newcommand{\downbundle}{{\rotatebox[origin=c]{-90}{$\bundle$}}}

\newcommand{\bigover}[3]{
	 \begin{pmatrix}{\vphantom{f_f^f}#1} \\ {\vphantom{f_f^f}#2} \\ {\vphantom{f_f^f}#3} \end{pmatrix}
}
\newcommand{\littleover}[3]{
	 \begin{psmallmatrix}{\vphantom{f}#1} \\ {\vphantom{f}#2} \\ {\vphantom{f}#3} \end{psmallmatrix}
}
\newcommand{\sndover}[3]{
  \relax\if@display
	 \bigover{#1}{#2}{#3}
  \else
	 \littleover{#1}{#2}{#3}
  \fi
}

\allowdisplaybreaks
\begin{document}



\jnlPage{1}{00}
\jnlDoiYr{2023}
\doival{10.1017/xxxxx}

\title{A Fibrational Theory of First Order Differential Structures}

\begin{authgrp}
    \author{Matteo Capucci}
        \affiliation{Strathclyde University, United Kingdom,\\
        \email{matteo.capucci@strath.ac.uk}}
    \author{Geoffrey S. H. Cruttwell}
        \affiliation{Mount Allison University, Canada\\
        \email{gcruttwell@mta.ca}}
    \author{Neil Ghani}
        \affiliation{Strathclyde University, United Kingdom\\
        \email{neil.ghani@strath.ac.uk}}
    \author{ Fabio Zanasi}
        \affiliation{University College London, United Kingdom, and University of Bologna, OLAS Team (INRIA), Italy.\\
        \email{f.zanasi@ucl.ac.uk}}
\end{authgrp}

\history{(Received xx xxx xxx; revised xx xxx xxx; accepted xx xxx xxx)}

\begin{abstract}
    We develop a categorical framework for reasoning about abstract properties of differentiation, based on the theory of fibrations. Our work encompasses the first-order fragments of several existing categorical structures for differentiation, including cartesian differential categories, generalised cartesian differential categories, tangent categories, as well as the versions of these categories axiomatising reverse derivatives. We explain uniformly and concisely the requirements expressed by these structures, using sections of suitable fibrations as unifying concept. Our perspective sheds light on their similarities and differences, as well as simplifying certain constructions from the literature.\end{abstract}


\maketitle

\section{Introduction}
\label{sec:intro}






In the past few years, the widespread adoption of quantitative mathematical modelling, in areas such as probabilistic programming and machine learning, has determined a growing interest in the foundations of differential calculus. In particular, since the invention of differential categories \cite[]{blute_differential_2006}, much work has been devoted to categorical axiomatisations of differentiation. These include  Cartesian differential categories \cite[]{blute_cartesian_2009}, generalised Cartesian differential categories \cite[]{cruttwell_cartesian_2017}, the (re)discovery of tangent categories \cite[]{Rosicky84} in \cite[]{cockett_differential_2014}, as well as ``reverse'' versions of many of these structures \cite[]{cockett_reverse_2020, cruttwell_gallagher_lemay_pronk_2022}. These approaches have provided categorical foundations for aspects of differential linear logic \cite[]{Ehrhard18} and the differential $\lambda$-calculus \cite[]{EhrhardR03}, as well as giving semantics for algorithms in machine learning \cite[]{Crutwell21} and automatic differentiation \cite[]{Alvarez-Picallo23}, and extending key ideas of differential geometry to more general settings (for instance, connections have been generalised so that one may speak of curvature within arbitrary tangent categories).

As the categorical study of differentiation reaches a mature stage, it is natural to ask to what extent the various approaches may be reconciled within a common framework. This question does not only answer a need for unification, but also a desire to better understand the fundamental concepts underpinning differential structures. Indeed, the axioms chosen for these categories often provide a fine-grained description of differential operators, whose reconciliation with more familiar categorical structures is unclear. For instance, take the axiom $D[\pi_1] = \pi_1\circ\pi_1$ of cartesian differential categories \cite[]{blute_cartesian_2009}: why is it needed, and what are its consequences?

Providing transparent axiomatisations and abstract definitions of differentiation is a key motivation for this paper.
More broadly, with the increasing number of categorical structures for differentiation, it has become important to systematise their underlying assumptions, classify their features, and understand their consequences.

Our central thesis is that significant theoretical and practical
advances can be made by recasting categorical differentiation through a fibrational prism.
This is a rather natural step.
Fibrations are a categorical model of dependency, and hence formalise a core idea of differential calculus: that the tangent space is a thickening of the base space.
In fact, our thesis is not new, in that fibrations already arise in the literature on categorical differential structures.
However, our approach is more radical, as it suggests fibrations are the fundamental concept around which categorical differential structures can be developed.
This paper provides concrete evidence for such claim, by showing that a unified fibrational framework is able to encompass a large variety of first-order differential structures --- where by first-order we mean those structures involving only one use of a differential operator to state.
For instance, within the axiomatisations of cartesian differential categories (CDCs) \cite[]{blute_cartesian_2009} and of reverse cartesian derivative  categories (RCDCs) \cite[]{cockett_reverse_2020}, the first-order theory consists of the first five axioms, thus excluding those axioms stating the symmetry of partial derivatives and the linearity of the derivative in its second component.

One might ask why first order differential structures are worthy as an object of independent study. Even though a complete account of differential algebra clearly requires higher order structure ---and we will return to this in Section~\ref{sec:conclusions} --- the first order fragment suffices for various applications of interest. In particular, differentiation is of utmost importance in programming, especially in the context of machine learning. The interplay between the type structure of functions and their derivative raise questions that remain unresolved even in the first-order setting. For example, the derivative of a function $f \colon X+Y \to Z$ has type $(X+Y) \times (X+Y) \to Z$ when solved in the first-order theory of CDCs. However, we might prefer the more accurate type $(X \times X) + (Y \times Y) \to Z$, which captures the intuition that, if our input comes from the type $X$, and we make a small change to it, then we remain in $X$. A systematic mathematical foundation of first-order differential structures is a necessary step towards accounting for type structure in the presence of differential structure.

In developing our framework, the main technical ingredient is the use of {\em sections of a fibration} as the core abstraction driving differential structures---this contrasts with the operator based approach of CDCs, RCDCs, and generalised CDCs, and with the functor based approach of tangent categories.
By tuning the underlying categories and their properties, we will show that sections of fibrations are able to capture a wide range of different first-order theories, encompassing CDCs, RCDCs, generalised CDCs, change actions and tangent categories. The structures involved in this unification turn out to be surprisingly simple. Furthermore, the approach brings a novel, clarifying perspective on constructions that in the standard theory of differential structures are laborious to obtain, or hard to describe directly. Examples are the canonical construction of a CDC from an RDC (see~\cref{thm:cdc_from_rdc} below), the notion of reverse tangent category (see~\cref{def:revtangentcat} below), and the cofree CDC over a tangent category (\cref{th:cdc-of-lin-objs}).

The paper is structured as follows.~\cref{sec:preliminaries} gives background preliminaries on cartesian and (various versions of) additive categories, as well as on fibrations.~\cref{sec:fods} characterises a variety of first-order differential structures as sections of a fibration. In~\cref{sec:diff-objs} we use our framework to conduct an abstract study of linearity and differential objects in first-order differential structures.
Our conclusions, in~\cref{sec:conclusions}, include a discussion of how our approach may be extended to higher order axioms for differentiation.

\paragraph*{Notation and prerequisites.}
We assume the reader is familiar with basic category theory. We write the composite of maps $f \colon X \to Y$ and $g \colon Y \to Z$ as $g \circ f : X \to Z$, which we sometimes abbreviate as $gf$. The identity map on an object $A$ is written $\id_A$. We write $|\cat B|$ for the discretisation of $\cat B$, ie. the category whose objects are those of $\cat B$ and the only morphisms are the identities.



\section{Preliminaries}
\label{sec:preliminaries}

The fundamental idea behind derivation is that of \emph{linear approximation}.
Derivation can be broadly interpreted as an operation that shows how linear variations in function's input map to linear changes in its output.
This simple idea actually rests on a scaffolding of more fundamental ones. What is a function? What is a variation? What does linear mean?

In this section we define the various bits of structure necessary to state a definition of derivative in a categorical setting. Our perspective is that a derivative operates over the morphisms of a cartesian category (of \emph{non-linear maps}), the variations over an objects are given as a fibration (of \emph{bundles}) over said category, and linearity---which really means additivity, is extra structure on the fibers of said category (making them \emph{semi-additive} or \emph{left-additive} categories).

\subsection{Cartesian Structures}
\label{sec:cart-add-structs}

We recapitulate some basic definitions and notation about categories with finite products.

\begin{definition}\label{def:cartesiancat}
    We call \textbf{cartesian} a category with finite products.
    A \textbf{product-preserving functor} $F:\cat A \to \cat B$ is a functor that maps finite product cones to product cones.
    A \textbf{cartesian natural transformation} between product-preserving functors $F$ and $G$ is a natural transformation $\alpha:F \to G$ such that $\alpha_{X \times Y} = \alpha_X \times \alpha_Y$.
    We write $\CartMonCat$ for the \textbf{2-category of cartesian categories}, product-preserving functors, and cartesian natural transformations.
\end{definition}

We will write $A \nfrom{\pi_A} A \times B \nto{\pi_{B}} B$ for the product projections, or simply $A \nfrom{\pi_0} A \times B \nto{\pi_{1}} B$ if we want to gloss over the name of objects. The pairing of maps $f \colon A \to B$ and $g \colon A \to C$ is written $\langle f , g\rangle \colon A \to B \times C$.

We will also need the following definitions:

\begin{definition}\label{def:prod-reflecting}
    A functor $F: \cat A \to \cat B$ is \textbf{product-reflecting} if (i) $FA$ is terminal implies $A$ is terminal, and (ii) if $FA \nfrom{F\pi_A} FP \nto{F\pi_{B}} B$ is a product cone in $\cat B$, then $A \nfrom{\pi_A} P \nto{\pi_{B}} B$ was already a product cone in $\cat A$.
    A functor between cartesian categories is \textbf{product-creating} if it preserves and reflects products.
\end{definition}

\begin{remark}
    Let us excuse for the slight abuse of terminology we are adopting here: our product-preserving (resp. reflecting, creating) functors are really only \emph{finite-products}-preserving.
\end{remark}

\subsection{Additive structures}
If $\cat B$ has finite products, one may form the category of commutative monoids internal to $\cat B$, which we denote by $\CMon(\cat B)$. $\CMon(\cat B)$ inherits products from $\cat B$. It turns out $\CMon(\cat B)$ also has coproducts and these coincide with the products.
Categories with products and coproducts which coincide are said to have \emph{biproducts}: we now recall this notion.

\begin{definition}
    A category $\cat A$ \textbf{has finite biproducts} if it admits all finite products and coproducts, and these coincide.
    Concretely, such a category has a zero object $0$, which admits universal maps from and to every object of $\cat A$, and for every $A,B:\cat A$ there is an object $A \oplus B$ which is both a product and coproduct of $A$ and $B$.
\end{definition}

A category $\cat B$ with biproducts is automatically enriched in commutative monoids. For each pair of objects $A,B$ in $\cat B$, the zero in $\cat B(A,B)$ is defined as:
\begin{equation}
\label{eq:cmon_enrichment0}
    0_{A,B} :=  A \nto{!_A} 0 \nto{?_B} B,
\end{equation}
while sum is defined as:
\begin{equation}
\label{eq:cmon_enrichment1}
    \forall f,g:A \to B : \cat B, \quad f+g := A \nto{\Delta_A} A \oplus A \nto{f \oplus g} B \oplus B \nto{\nabla_B} B.
\end{equation}
Here $\Delta_A$ and $\nabla_B$ are, respectively, the diagonal and codiagonal induced by the product and coproduct structures on $A \oplus B$. Note~\eqref{eq:cmon_enrichment0} and \eqref{eq:cmon_enrichment1} are independent of the specific choice of zero object and coproduct respectively.
This definition motivates working with cartesian categories instead of arbitrary symmetric monoidal categories.

The category $\CMon(\cat B)$ has finite biproducts and, moreover, it has a universal property:

\begin{lemma}
\label[lemma]{lemma:cofree-semiadd}
    Let $\cat B$ be a category with finite products.
    The category $\CMon(\cat B)$ of commutative monoids in $\cat B$ and homomorphisms between them is a category with finite biproducts, and in fact the cofree one on $\cat B$.
    Furthermore, if $F: \cat B \to \cat B'$ is a product-preserving functor, $F$ extends to a product preserving-functor $\CMon(F) : \CMon(\cat B) \to \CMon(\cat B')$.
\end{lemma}
\begin{proof}
    For every $A,B:\CMon(\cat B)$, $A \times B$ is a biproduct: its coproduct injections are $\iota_A := {A \nlongto{\langle \id_A, 0_B\rangle} A \times B}$ and $\iota_B := B \nlongto{\langle 0_A \times \id_B \rangle} A \times B$, where $0$ denotes the zero maps of the commutative monoids.
    The fact that $\CMon(\cat B)$ is cofree can be verified by following an argument similar to that of~\cite{fox_coalgebras_1976}.
    The last part follows from the fact that product preserving functors preserve commutative monoid structure.
\end{proof}

Categories with finite biproducts are also called \emph{semi-additive}. We denote the 2-category of categories with finite biproducts, (bi)product-preserving functors and cartesian natural transformations as $\SemiaddCat$.

Having all finite biproducts is a strong requirement. A weaker alternative, introduced in
\cite[]{blute_cartesian_2009}, is that of~\emph{cartesian left-additive categories} (\emph{CLA}s), which are cartesian categories $\cat B$ such that every object of $\cat B$ is assigned a chosen commutative monoid structure.

\begin{definition}[\bf Cartesian Left-Additive Category]
\label[definition]{def:CLA}
    A \textbf{cartesian left-additive category (CLA)} is a pair $(\cat B, M)$ where $\cat B$ is a category with finite products and a functor $M \colon |\cat B| \to \CMon(\cat B)$ associating to each object a commutative monoid such that:
    \begin{enumerate}
        \item the carrier of $M(X)$ is $X$;
        \item $M(X \times Y) = M(X) \times M(Y)$.
    \end{enumerate}
\end{definition}

Note we do not require maps in $\cat B$ to preserve the chosen commutative monoid structures. CLAs provide a setting which is in-between that of a category with biproducts and category with products. In particular, the product of a CLA category has some of the structure of a coproduct too.
As in the proof of Lemma~\ref{lemma:cofree-semiadd}, we can define injections $\iota_A := A \xrightarrow{\langle \id_A , 0_B \rangle} A \times B$ and $\iota_B := B \xrightarrow{\langle 0_A , \id_B \rangle} A \times B$, where $0_A$  and $0_B$ denote the zeros of the respective commutative monoids. We can also form `cotuples', as follows: given $f:A \to C$ and $g: B \to C$, we can define $[f,g] \colon A \times B \to C$ by $[f,g] := +_C \circ (f \times g)$ where $+_C$ is the addition on $C$.

Like categories with finite biproducts, every hom-set of a CLA $\cat B$ is a commutative monoid.
However, only left composition preserves this structure, i.e.~only $(f+g)h = fh + gh$. Maps preserving the monoid structure completely are called \emph{additive}.

\begin{definition}[\bf Additive Maps]
\label[definition]{def:additive-maps}
    In a CLA $\cat B$, a map $f$ such that, for every $g,h$, we have $f(g+h) = fg + fh$ and $f0 = 0$, is called \textbf{additive}.
    The wide subcategory of additive maps of $\cat B$, written $\cat B^+$, is a category with biproducts whose inclusion $i:\cat B^+ \to \cat B$ creates products.
\end{definition}

The reason $i$ creates products is that, even though we did not mention maps at all in~\cref{def:CLA}, the fact that projections and diagonals of a cartesian category are natural makes them automatically additive. We will broadly adopt this terminology, calling `additive' maps that preserve some commutative monoid structure, even when not in the setting of a CLA.

It remains to introduce a notion of functor preserving CLA structure.

\begin{definition}
    A functor $F:\cat A \to \cat B$
    between CLAs $(\cat B, M)$ and $(\cat B', M')$ is a \textbf{CLA functor} if it preserves products and the chosen monoids, i.e.~such that $M'\circ |F| = \CMon(F) \circ M$.
\end{definition}

Thus a CLA functor is a product-preserving functor which respects the choice of monoids.
By virtue of this, a CLA functor $F$ satisfies the expected equations $F(f+g) = Ff + Fg$ and $F(0) = 0$ and preserves additive maps.

\begin{definition}\label{def:CLACat}
	CLAs, CLA functors and natural transformations form a 2-category, which we denote by $\CLACat$.
\end{definition}

One can take a more structural approach to CLA structures, by axiomatizing instead the subcategory of additive maps:

\begin{definition}
\label[definition]{def:equip-of-add-maps}
	Let $\cat B$ be a category with finite products.
    An \textbf{equipment of additive maps} on $\cat B$ consists of an bijective-on-objects, product-creating functor from a semi-additive category $\cat A$:
    \begin{equation}
        i: \cat A \longto \cat B.
    \end{equation}
	We say an object (map) in $\cat B$ is \textbf{additive} iff it's the image of some object (map) in $\cat A$.
\end{definition}

\begin{lemma}
    Let $i: \cat A \to \cat B$ be an equipment of additive maps. Then, in $\cat B$:
    \begin{enumerate}
        \item projections out of a finite product are additive,
        \item diagonals are additive.
    \end{enumerate}
\end{lemma}
\begin{proof}
	Follows from the definition of additive and the fact $i$ is product-reflecting.
\end{proof}

Observe that maps in $\cat A$ are morphisms of commutative monoids by virtue of living in a category with biproducts, and since $i$ preserves products, all additive objects are commutative monoids and all additive maps preserve this structure.
In fact we have:

\begin{lemma}
\label[lemma]{lemma:choice-of-cmons}
    An equipment of additive maps $i:\cat A \to \cat B$ on a cartesian monoidal category $\cat B$ induces a map $m : \cat A \to \CMon(\cat B)$, that we call its \textbf{choice of commutative monoids}:
    \begin{equation}
		\begin{tikzcd}[ampersand replacement=\&,sep=scriptsize]
			\& {\CMon(\cat B)} \\
			{\cat A} \& {\cat B}
			\arrow["i"', from=2-1, to=2-2]
			\arrow["{\varepsilon_{\cat B}}", from=1-2, to=2-2]
			\arrow["{\exists!\,m}", dashed, from=2-1, to=1-2]
		\end{tikzcd}
    \end{equation}
\end{lemma}
\begin{proof}
	Direct consequence of~\cref{lemma:cofree-semiadd}.
\end{proof}


Therefore equipments of additive maps give a way to fit a category of interest $\cat B$ with commutative monoids.
Compared to CLAs or supplies~\cite[]{fong_supplying_2020}, this choice is functorial and thus easier to interoperate with other categorical constructions---indeed, this approach will shine in the proof of~\cref{th:cdc-of-lin-objs}.

However, observe that while all additive maps preserve the monoid structure on additive objects, the converse is not true: a map between additive objects in $\cat B$ which preserves the monoid structure isn't necessarily additive.
In other words, $m$ might fail to be full.

\begin{definition}
\label[definition]{def:saturated-equip}
    An equipment of additive maps $i:\cat A \to \cat B$ is \textbf{saturated} if its associated choice of commutative monoids is fully faithful:
	\begin{equation}
		\begin{tikzcd}[ampersand replacement=\&,sep=scriptsize]
			\& {\CMon(\cat B)} \\
			{\cat A} \& {\cat B}
			\arrow["i"', from=2-1, to=2-2]
			\arrow["{\varepsilon_{\cat B}}", from=1-2, to=2-2]
			\arrow["{\ff\ m}", dashed, from=2-1, to=1-2]
		\end{tikzcd}
	\end{equation}
\end{definition}

Thus an equipment of additive maps is saturated if we can check if a map between additive objects is additive by simply checking whether it's a morphism of the chosen commutative monoids; and if an additive map is the image of at most one map in $\cat A$.

\begin{lemma}
\label[lemma]{lemma:CLA-charac}
    A saturated equipment of additive maps is equivalent to a CLA structure.
\end{lemma}
\begin{proof}
    Let $\cat B$ be a cartesian category.
    Given a choice of commutative monoids for $\cat B$ as in~\cref{def:CLA}, one can produce a semi-additive category of additive maps $\cat B^+$, whose inclusion is bijective-on-objects and creates products, as noticed in~\cref{def:additive-maps}.
    By definition this is a saturated equipment of additive maps.

    \emph{Vice versa}, given a saturated equipment consider the object part $|m|:|\cat A| \to |\CMon(\cat B)|$ of the choice of monoids functor $m$. Since $\phi:|\cat A| \cong |\cat B|$, consider the functor $\phi^{-1}|m|:|\cat B| \to |\CMon(\cat B)|$. It is clearly identity-on-objects and product preserving, thus satisfying~\cref{def:CLA}.

\end{proof}

\subsection{Categories with bundles}
\emph{Categories with bundles} are similar to categories with display maps~\cite[§8.3]{taylor_practical_1999} in intent, but differ in that we require bundle maps to be closed under composition.

\begin{definition}[Category with Bundles]
\label[definition]{def:disp-map-cat}
    A \textbf{category with bundles} $(\cat B, \cat F)$ is a category together with a distinguished subcategory $\cat F$, whose morphisms we call \textbf{bundles}, with the property that all pullbacks of bundles against arbitrary maps exist and are again bundles.
    A category with bundles is \textbf{cartesian} when both $\cat  B$ and $\cat F$ admit all finite products (meaning $\cat F$ contains all product projections).
    A \textbf{functor of categories with bundles} between $(\cat B, \cat F)$ and $(\cat B', \cat F')$ is a functor $F:\cat B \to \cat B'$ that preserves bundles and pullbacks thereof, i.e.~when $p:E \to B$ is a bundle then so is $F(p)$ and when $f:A \to B$ is any map then $F(f^*p) = Ff^*(Fp)$.
    Categories with bundles, functors of categories with bundles, and natural transformations, form a 2-category $\BundleCat$.
\end{definition}

\begin{example}
    The archetypal example of a category with bundles is that of real smooth manifolds $\Smooth$: there, not all pullbacks exist, but submersions (surjective maps with surjective differentials) can be pulled back along any smooth map and are closed under composition, thus forming a category with bundles.
\end{example}


\subsection{Fibrations}
\label{sec:fibs}

We recall some basic definitions of the theory of fibrations, directing the reader to~\cite{jacobs_categorical_1999} and~\cite{streicher_fibered_2023} for a more comprehensive treatment of the subject. 

\begin{definition}
\label[definition]{def:cartesian-vertical}
    Let $q:\cat E \to \cat B$ be a functor.
    A map $f: X \to Y$ in $\cat E$ is said to be \textbf{cartesian} if for all maps $g : E' \to Y$ in $\cat E$ and $h : q(E') \to q(X)$ in $\cat B$ such that $h$ factors $q(g)$ through $q(f)$, there is a unique map $\varphi : E' \to X$ such that $q(\varphi) = h$ and $g = f \varphi$.
    \begin{equation}
        \begin{tikzcd}[ampersand replacement=\&, sep=scriptsize]
            {\forall E'} \\
            \& {X} \& Y \\
            {q(E')} \& q(X) \& {q(Y)}
            \arrow["q(f)", from=3-2, to=3-3]
            \arrow["{f}", from=2-2, to=2-3]
            \arrow["{\forall g}", curve={height=-12pt}, from=1-1, to=2-3]
            \arrow["{q(g)}"', curve={height=18pt}, from=3-1, to=3-3]
            \arrow["{\exists!\,\varphi}"', dashed, from=1-1, to=2-2]
            \arrow["{\forall h}", from=3-1, to=3-2]
        \end{tikzcd}
    \end{equation}
    In this case we also say that $f$ is the \textbf{cartesian lift} of $q(f)$.
\end{definition}

\begin{remark}
    To avoid confusion, we note cartesian maps and cartesian products have nothing to do with each other, except being both named after Descartes.
\end{remark}

Keeping $q$ fixed as above, we say an object $E$ of $\cat E$ is said \textbf{to be over} an object $B$ of $\cat B$ when $q(E) = B$, and same for arrows. A map $v$ in $\cat E$ is called $q$-\textbf{vertical} (or simply vertical, if $q$ is clear from context) if it is over an identity map.

Given an object $B$ of $\cat B$, we define the \textbf{fibre} $\cat E_B$ to be the category whose objects are the objects of $\cat E$ which are over $B$, and whose maps are those over $\id_B$.

\begin{definition}
\label[definition]{def:fibration}
     A \textbf{fibration} is a functor $q : \cat E \to \cat B$ for which every map of the form $f:B \to q(Y)$ admits a cartesian lift $f_Y:f^*Y \to Y$.
\end{definition}

We call the object $f^*Y$ the \textbf{reindexing} of $Y$ along $f$.

We assume that all our fibrations are \textbf{cloven}, meaning the lifts are functorially chosen.

\begin{remark}
\label{rmk:lifts-are-pbs}
    With this structure, the total category of a fibration $q$ comes equipped with a canonical factorization system in which the left class is comprised of the $q$-vertical maps and the right class of the $q$-cartesian ones.
    Consequently, each map $h : E' \to E$ in the total category factors as a map $h^v$ in the fibre over $E'$ followed by a cartesian map $h^c$.
    The universal property of cartesian maps then says that vertical maps can always be pulled back along cartesian ones, and the resulting pullback square has cartesian and vertical maps on opposite sides.
\end{remark}

\begin{example}[Simple fibration]
\label{ex:simple_fib}
    Let $\cat B$ be a category with finite products.
    Define $S(\cat B)$ to be the category whose objects are pairs of objects ${A' \choose A}$ of $\cat B$ and whose morphisms are pairs of morphisms:
    \begin{equation}
    \label{eq:simple-fib}
        {f^\flat \choose f} : {A' \choose A} \chartto {B' \choose B} \qquad\text{ where }\qquad \begin{matrix}
            f^\flat : A \times A' \to B'\\
            f:A \to B
        \end{matrix}
    \end{equation}
The composite of ${f^\flat \choose f} : {A' \choose A} \chartto {B' \choose B}$ and ${g^\flat \choose g} : {B' \choose B} \chartto {C' \choose C}$ is ${h \choose gf} : {A' \choose A} \chartto {C' \choose C}$ where $h := A \times A' \xrightarrow{\Delta_A \times \id_{A'}} A \times A \times A' \xrightarrow{f \times f^\flat} B \times B'\xrightarrow{g^\flat} C'$.   The simple fibration $\simple_{\cat B} : S(\cat B) \to \cat B$ maps ${A' \choose A}$ to $A$ and ${f^\flat \choose f} : {A' \choose A} \chartto {B' \choose B}$ to $f:A \to B$.
    A cartesian morphism is one of the form ${\pi_{B'} \choose f} : {B' \choose A} \chartto {B' \choose B}$ where $\pi_{B'}$ is a product projection.
    A vertical morphism is one for which the bottom map is the identity.

(Note: another way to view the simple fibration is as the full subcategory of the codomain fibration consisting of projection maps).
\end{example}

\begin{example}[Simple lens fibration]
\label{ex:lens_fib}
    Let $\cat B$ be a category with finite products.
    Define $L(\cat B)$ to be the category whose objects are pairs of objects ${A' \choose A}$ of $\cat B$ and whose morphisms are pairs of morphisms as follows:
    \begin{equation}
        \label{eq:lensfib}
        {f^\flat \choose f} : {A' \choose A} \chartto {B' \choose B} \qquad\equiv\qquad \begin{matrix}
            f^\flat : A \times B' \to A'\\
            f:A \to B
        \end{matrix}
    \end{equation}
    The lens fibration $\lens_{\cat B} : L(\cat B) \to \cat B$ maps ${A' \choose A}$ to $A$ and ${f^\flat \choose f} : {A' \choose A} \chartto {B' \choose B}$ to $f:A \to B$.
    The name is due the fact that pairs of morphisms as in \eqref{eq:lensfib} are also known as \emph{lenses}, and are widely used to model operations with a reverse mode, in contexts such as functional programming~\cite{lenses-haskell}, databases~\cite{Bancilhon-lensdatabases}, game theory~\cite{Hedges17}, and machine learning~\cite{Crutwell21}.
\end{example}

\begin{example}[Codomain fibration]
\label{ex:cod_fib}
    Let $\cat B$ be a category with pullbacks. We write $\cat B^{\downarrow}$ for the arrow category of $\cat B$, whose objects are morphisms in $B$, and a morphism from $f : A \to B$ to $g : C \to D$ is a commuting square
    \begin{equation}\label{eq:comparrowcat}
    	   \begin{tikzcd}[ampersand replacement=\&,sep=small]
                {A} \& {C} \\
                B \& {D}
                \arrow["g", from=1-2, to=2-2]
                \arrow["h"', from=2-1, to=2-2]
                \arrow["f"', from=1-1, to=2-1]
                \arrow[from=1-1, to=1-2]
            \end{tikzcd}
    \end{equation}
    The codomain functor $\cod:\cat B^{\downarrow} \to \cat B$ maps an object $f : A \to B$ to $B$, and a morphism as in~\eqref{eq:comparrowcat} to $h : B \to D$. $\cod$ is a fibration, where cartesian maps are defined by pullback.
    Indeed, the cartesian maps in $\cat B^\downarrow$ are pullback squares, while vertical maps are square whose bottom is an identity:
    \begin{equation}
        \begin{matrix}
            \begin{tikzcd}[ampersand replacement=\&,sep=small]
                {f^*E'} \& {E'} \\
                B \& {B'}
                \arrow["p", from=1-2, to=2-2]
                \arrow["f"', from=2-1, to=2-2]
                \arrow[from=1-1, to=2-1]
                \arrow[from=1-1, to=1-2]
                \arrow["\lrcorner"{anchor=center, pos=0.125}, draw=none, from=1-1, to=2-2]
            \end{tikzcd}
            &\hspace*{5ex}&
            \begin{tikzcd}[ampersand replacement=\&,sep=small]
                E \& {E'} \\
                B \& {B}
                \arrow["p", from=1-2, to=2-2]
                \arrow[Rightarrow, no head, from=2-1, to=2-2]
                \arrow["q"', from=1-1, to=2-1]
                \arrow["v", from=1-1, to=1-2]
            \end{tikzcd}
            \\
            \text{cartesian square}
            &&
            \text{vertical square}
        \end{matrix}
    \end{equation}
\end{example}

In the sequel we will relate the codomain fibration to tangent categories, see~\cite{cockett_differential_2014}. However, important examples of tangent categories do not have all pullbacks, thus a refinement of the codomain fibration is required, which applies to bundles (see Definition~\ref{def:disp-map-cat}). Every category with bundles gives rise to a fibration, which is intuitively the codomain fibration of~\cref{ex:cod_fib} modified to only require the ability to reindex bundles:

%

\begin{example}[Fibration of bundles]
\label{ex:bun_fib}
    Let $(\cat B, \cat F)$ be a category with bundles.
    Define $\cat B^\downbundle$ to be the subcategory of the arrow category $\cat B^\downarrow$ of $\cat B$ whose objects are restricted to maps in $\cat F$, while still allowing any map in $\cat B$ in the squares comprising the morphisms of $\cat B^\downbundle$.
    Since by assumption maps in $\cat F$ are stable under pullback, $\cod_{\cat B} : \cat B^\downbundle \to \cat B$ is a fibration.
\end{example}

\begin{example}[Simple fibration as a fibration of bundles]
    A category with finite products $\cat B$ defines a category with bundles where the bundles are the product projections.
    The fibration of bundles associated to such a structure can be seen to be equivalent to the simple fibration over $\cat B$, introduced in~\cref{ex:simple_fib}.
\end{example}


We will also need the notion of a fibration with fibred products/biproducts. This definition instantiates the more general one, for fibred finite limits, given e.g.~in~\cite{jacobs_categorical_1999}.

\begin{definition}\label{def:fibredproducts}
    A fibration has \textbf{fibred products} if each fibre has products and reindexing preserves them.
    Likewise, a fibration has \textbf{fibred biproducts} if each fibre has biproducts and reindexing preserves them.
\end{definition}

\begin{example}
    The fibres of the codomain fibration have products given by pullback.
    Since reindexing is also given by pullback, it follows the codomain fibration has fibred products.
    Similarly, the simple fibration admits fibred products given by ${B \choose A} \times_A {B' \choose A} = {B \times B' \choose A}$.
\end{example}

We recall the following easy fact, appearing as~\cite[Lemma~1.8.4]{jacobs_categorical_1999}:

\begin{lemma}
    Having fibred products is stable under change of base, that is: if $q:\cat E \to \cat B$ is a fibration with fibred products, and $F_0:\cat A \to \cat B$ is any functor, then the pullback $F_0^*q$ has fibred products.
\end{lemma}


\begin{definition}[{\bf Fibred Functors}]
    Given fibrations $q: \cat E \to \cat B$ and $q':\cat E' \to \cat B'$,
    a \textbf{fibred functor} $(F_0,F_1):q \to q'$ consists of functors $F_0:\cat B \to \cat B'$ and $F_1 : \cat E \to \cat E'$ such that $F_1$ preserves cartesian morphisms and the following square commutes:
    \begin{equation}
        \begin{tikzcd}[ampersand replacement=\&,sep=scriptsize]
            {\cat E} \& {\cat E'} \\
            {\cat B} \& {\cat B'}
            \arrow["{F_0}", from=2-1, to=2-2]
            \arrow["{F_1}", from=1-1, to=1-2]
            \arrow["q"', from=1-1, to=2-1]
            \arrow["{q'}", from=1-2, to=2-2]
        \end{tikzcd}
    \end{equation}
    Fibrations and fibred functors form a category $\Fib$.
    For every fixed category $\cat B$, the subcategory of $\Fib$ spanned by fibrations over $\cat B$ and fibred functors where $F_0$ is the identity is denoted $\Fib(\cat B)$.
\end{definition}

We conclude by recalling the dual of a fibration. It will be useful when presenting reverse differential categories as a `reverse' construction of (forward) differential categories.

\begin{definition}
    Let $q:\cat E \to \cat B$ be a (cloven) fibration.
    Its \textbf{dual} or \textbf{fibrewise opposite} $q^{\vee}:\cat E^{\vee} \to \cat B$ has as fibres the opposite of those of $q$.
    Explicitly, define $\cat E^{\vee}$ as the category with the same objects as $\cat E$, and maps $\cat E^{\vee}(X,Y)$ consisting of spans
    \begin{equation}
    \label{eq:dual-fib-spans}
        \begin{tikzcd}[ampersand replacement=\&, sep=scriptsize]
            \& {f^*Y} \\
            X \&\& Y
            \arrow["f", from=1-2, to=2-3]
            \arrow["{f^\sharp}"', from=1-2, to=2-1]
        \end{tikzcd}
    \end{equation}
    where $f$ is cartesian and $f^\sharp$ is vertical.
    The identity span has this form, and they compose by pullback (Remark~\ref{rmk:lifts-are-pbs}).
    Then $q^\vee$ sends the pair ${f^\sharp \choose f}$ down to $f$, and this completes the defintion.
\end{definition}

More details can be found in~\cite[Definition~1.10.11]{jacobs_categorical_1999}.

\begin{example}
    The fibration of simple lenses of~\cref{ex:lens_fib} is the opposite of the simple fibrations~\cref{ex:simple_fib}, as can be easily verified by comparing~\eqref{eq:simple-fib} and~\eqref{eq:lensfib}
\end{example}

\begin{example}[Dependent lenses aka polynomial functors aka containers]
    In~\cite{spivak_generalized_2019}, the authors extend this terminology to all opposite fibrations, calling the spans of~\eqref{eq:dual-fib-spans} \emph{generalised lenses}.
    For instance, the fibrewise opposite of the codomain fibration of $\Set$ yields so-called \emph{dependent lenses}, which are equivalent to polynomial functors over $\Set$~(\cite{gambino_polynomial_2013,spivak_poly_2020}) and containers~(\cite{abbott_containers_2005}).
    Notice all three concepts generalise in different directions.
\end{example}

For our developments, we will need a theory of fibrations which coheres with categorical structures such as products, CLA structure, bundles, etc.
The principled way to execute this programme is by defining fibrations in a more general setting, which is parametric to such structure.
As noted in~\cite{Street-fib2cat}, the definition of a fibration, usually refereed to categories, functors, and natural transformations, can in fact be stated in any 2-category.
In our cases of interest, namely the 2-categories $\CartMonCat$ (Definition~\ref{def:cartesiancat}), $\CLACat$ (\cref{def:CLA}), $\BundleCat$ (\cref{def:disp-map-cat}), fibrations are fibrations (i) whose domain and codomain possess the structure in question; and (ii) the fibration and cartesian lift preserve that structure.\footnote{This is due to the fact that in a 2-category $\dblcat K$ which admits comma objects, a fibration is a 1-cell $q:E \to B$ equipped with a 1-cell $\ell : B/q \to E$ which chooses cartesian lifts.
Hence if $\dblcat K$ is a 2-category of structured categories, both $q$ and $\ell$ are maps compatible with such structure. See~\cite{Street-fib2cat} for details.}

For example, for $\CartMonCat$ we have the following notion.

\begin{definition}[{\bf Fibration of Cartesian Categories}]
\label[definition]{def:fib-of-cart-cats} Let $\cat E$ and $\cat B$ be cartesian categories.
    A \textbf{fibration of cartesian categories} is a fibration (in the sense of~\cref{def:fibration}) $q: \cat E \to \cat B$ such that
    \begin{enumerate}
        \item $q$ preserves products,
        \item if $k, k'$ are cartesian lifts of maps $q(k)$, $q(k')$ in $\cat B$, then $k \times k'$ is the cartesian lift of $q(k) \times q(k')$.
    \end{enumerate}
\end{definition}

The following fact (see~\cite{hermida_properties_1999}) will help us prove that certain fibrations are cartesian fibrations, by relating \emph{global} products (taken in the total category) with \emph{local} ones (taken in a fibre).

\begin{lemma}
\label[lemma]{lem:fibprod-cartfib}
    Let $q: \cat E \to \cat B$ be a fibration (of categories) (Definition~\ref{def:fibredproducts}).
    If $\cat B$ is cartesian and $q$ has fibred products, then $\cat E$ is cartesian and $q$ is in fact a fibration of cartesian categories.
\end{lemma}

\begin{example}
    When $\cat B$ has all finite limits, thus admitting both pullbackas and products, both $\cat B^\downarrow$ and $\cat B$ have finite products, and the codomain fibration $\cod:\cat B^\downarrow \to \cat B$ is a fibration of cartesian categories.
    In fact, clearly taking codomains preserve products strictly and products of pullback squares are still pullback squares (since limits commute with limits).
\end{example}

\begin{example}
    The simple fibration $\simple_{\cat B} : S(\cat B) \to \cat B$ associated to a category with finite products is a fibration of cartesian categories.
    Products in $S(\cat B)$ are computed `pointwise', meaning ${A' \choose A} \times {B' \times B} = {A' \times B' \choose A \times B}$, with projections onto, say, ${A' \choose A}$, given by:
    \begin{equation}
        {\pi_A' \choose \pi_A} : {A' \times B' \choose A \times B} \chartto {A' \choose A}
    \end{equation}
    Clearly $\simple_{\cat B}$ preserves products, and it's easy to verify the product of cartesian maps is cartesian.
\end{example}


Cartesian left-additive categories can also be extended to the fibrational world.
\begin{definition}[{\bf Fibration of CLA Categories}]
\label[definition]{def:fib-of-CLA-cats}
    A \textbf{fibration of CLA categories} is a fibration of cartesian categories $q: \cat E \to \cat B$ such that
    \begin{enumerate}
        \item $\cat E$ and $\cat B$ are CLA categories, and $q$ is a CLA functor,
        \item cartesian lifts of additive maps are additive,
        \item zero maps are cartesian, and the sum of cartesian lifts is still cartesian.
    \end{enumerate}
\end{definition}

\section{First-order Differential Structures}
\label{sec:fods}

In this section we expound how sections of certain fibrations can be used to model the `first-order' aspect of various differential structures. In this work we use the expression first-order in the analytic sense, indicating \emph{structure that can be defined without multiple applications of a differential combinator}.
So, for example, for cartesian differential categories, this refers to axioms CDC1-5 as introduced in \cite[\S 2.1]{blute_cartesian_2009}; for reverse derivative categories, this refers to RDC1-5, as introduced in \cite[\S 2.2]{cockett_reverse_2020}. Our exposition will highlight how various categorical structures for differentiation may be uniformly understood in this manner, culminating in the theory of first-order differential structures of~\cref{sec:general-fods}.

\subsection{Cartesian Differential Categories} \label{sec:CDCsections}

We first focus on the case of cartesian differential categories (CDCs), as introduced in~\cite{blute_cartesian_2009}. Note that, for this case, the fibrational treatment is not new, appearing already in~\cite[\S 2.4]{blute_cartesian_2009}. In the context of our exposition, it is instructive to recover it in steps. Our starting point is the simple fibration $\simple \colon S(\cat B) \to \cat B$ on a category $\cat B$, see Example~\ref{ex:simple_fib}. A \textbf{section} of such fibration is simply a functor $D \colon \cat B \to S(\cat B)$ which is a left inverse, that is, $\simple \circ D = \id_{\cat B}$. We say that $D$ is \textbf{stationary} whenever $DA = {A \choose A}$---which is a definition that only makes sense for sections of the simple fibration.

The data of a stationary section $D$ may be summarised as follows:
    \begin{equation}\label{eq:CDC}
        \begin{tikzcd}[ampersand replacement=\&,sep=scriptsize]
        	{S(\cat B)} \& {{A \choose A}} \& {{B \choose B}} \\
        	\\
        	{\cat B} \& A \& B
        	\arrow["{{\simple}}"', from=1-1, to=3-1]
        	\arrow["D"', curve={height=15pt}, dashed, from=3-1, to=1-1]
        	\arrow[""{name=0, anchor=center, inner sep=0}, "f", from=3-2, to=3-3]
        	\arrow["\delta f", shift left, from=1-2, to=1-3]
        	\arrow["f"', shift right, from=1-2, to=1-3]
        	\arrow[""{name=1p, anchor=center, inner sep=0}, phantom, from=1-2, to=1-3, start anchor=center, end anchor=center]
        	\arrow[shorten <=17pt, shorten >=17pt, maps to, from=0, to=1p]
        \end{tikzcd}
    \end{equation}

Now, let us assume $\cat B$ is a CLA category. The centrepiece of a CDC structure on $\cat B$ is an operator $\delta$ associating with any map $f \colon A \to B$ of $\cat B$ a map $\delta f \colon A \times A \to B$. Given $D$ as in~\eqref{eq:CDC}, we may readily define an operator $\delta$ of this type: its value on $f$ will be the second component of $Df$ (which we usually write above).

Conversely, suppose we are given an operator $\delta$ of the above type in $\cat B$. Then $\delta$ induces a stationary section $D \colon \cat B \to S(\cat B)$, defined as follows
\begin{equation*}
	    DA \ := \ {A \choose A} \qquad \qquad D(f : A \to A') \ := \ (f, \delta f) : {A \choose A} \to {A' \choose A'}
\end{equation*}
To form a CDC structure on $\cat B$, the operator $\delta$ also needs to satisfy certain axioms, listed as CDC1-7 in~\cite{blute_cartesian_2009}. As anticipated, we focus on the first-order axioms, namely CDC1-5. These five axioms are reported in equational form in the statement of Theorem~\ref{th:FODS-CDCs} below.

We shall discuss CDC1-5 step by step, via Lemmas below, expressing them as commutative diagrams in $\cat B$ to make their type more explicit and readable. Our goal is to show how each axiom on $\delta$ matches some property of the functor $D$, so that assuming the axiom holds amounts to assuming the functor has the property. The first lemma of this kind does not require any extra structure to be imposed on $D$.
\begin{lemma}\label{lemma:CDC1}
    A stationary section $D$ of the simple fibration $\simple \colon S(\cat B) \to \cat B$ as in~\eqref{eq:CDC} is exactly an operator $\delta$ making the diagrams $(\star)$ below commute for any $f \colon A \to B$ and $g \colon B \to C$ in $\cat B$:
    \begin{gather}
    	\xymatrix{ A \times A \ar@/^10pt/[rr]^{\delta \id_A} \ar@/_10pt/[rr]_{\pi_1} & {\scriptstyle (\star)} & A} \label{eq:CDC3lemma} \\
    	          \vcenter{
    \xymatrix{
C \ar@{}[drrr]|{(\star)} &&& \ar[lll]_{\delta g} \ar@/^/[drr]^{\pi_1} \ar@/^15pt/[drrr]^{\pi_0} B \times B & && \\
&& & \ar@/^15pt/[ulll]^{\delta (g \circ f)} A \times A \ar@/_15pt/[rrr]_{\delta f} \ar@{.>}[u]^{\langle \delta f, f \circ \pi_1\rangle} \ar[r]^{\pi_1} & A \ar[r]^{f} & B & B
}
}
 \label{eq:CDC5lemma}
    \end{gather}
\end{lemma}
\begin{proof}
    The proof just amounts to observing functoriality of $D$. Indeed, \eqref{eq:CDC3lemma} means exactly that $D$ maps identities in $\cat B$ to identities in $S (\cat B)$, while \eqref{eq:CDC5lemma} means exactly that $D$ preserves composition. That $D$ is a stationary section is by construction.
\end{proof}

Next, compatibility of $\delta$ with the left-additive structure of $\cat B$ requires that $D$ is also compatible with it.
\begin{lemma}\label{lemma:CDC2}
    Imposing commutativity of diagrams $(\star)$ below for any $f \colon A \to B$ and $g \colon A \to C$ in $\cat B$
    \begin{gather}
   \vcenter{
    \xymatrix{
   & B \times C \ar@/^5pt/[drrr]^{\pi_1} \ar@/^15pt/[drrrr]^{\pi_0} &&& & \\
    & \ar@/_15pt/[rrrr]_{\delta f} A \times A \ar@{{}{ }{}}@/^15pt/[u]|{(\star)}
     \ar@/^35pt/[u]^{\delta \langle f, g\rangle} \ar@{.>}[u]_{\langle \delta f, \delta g\rangle} \ar[rrr]^{\delta g} &&& C & B
    }
    }
 \label{eq:CDC4lemma}\\
\xymatrix{
A & \ar[l]^{\pi_A} A \times B &\ar[l]^</0.5cm/{\pi_1}
\ar@{{}{ }{}}@/_15pt/[ll]|{(\star)}
(A \times B) \times (A \times B)
\ar@{{}{ }{}}@/^15pt/[rr]|{(\star)}
\ar@/^25pt/[rr]^{\delta \pi_B} \ar@/_25pt/[ll]_{\delta \pi_A} \ar[r]_</0.4cm/{\pi_1} & A \times B \ar[r]_{\pi_B} & B
}
 \label{eq:CDC3lemma2}
 \end{gather}
    is equivalent to asking that $D$ preserves products.
\end{lemma}

\begin{proof}
    Immediate from the definition of products in $S(\cat B)$.
\end{proof}

\begin{lemma}
    Imposing commutativity of diagrams $(\star)$ below for any $f,g \colon A \to B$ in $\cat B$
    \begin{equation}
    \xymatrix{
    A \times A \ar@/^15pt/[rr]^{\delta f + \delta g}
    \ar@{{}{ }{}}@/^0pt/[rr]|{(\star)}
    \ar@/_15pt/[rr]_{\delta (f + g)}
 && B
    } \qquad
    \xymatrix{
    A \times A \ar@/^15pt/[rr]^{\delta 0_{A,B}}
    \ar@{{}{ }{}}@/^0pt/[rr]|{(\star)}
    \ar@/_15pt/[rr]_{0_{A \times A, B}} &&  B
    }\label{eq:CDC1lemma}
    \end{equation}
     is equivalent to asking that $D$ is not merely a product preserving functor but also a CLA functor.
\end{lemma}
\begin{proof}
    Immediate from the definition of CLA functor (Section~\ref{sec:cart-add-structs}).
\end{proof}

The purpose of these lemmas is to show how different axioms on $\delta$ correspond to requirements on $D$. Indeed, CDC1 is \eqref{eq:CDC1lemma}, CDC3 is given by \eqref{eq:CDC3lemma} and \eqref{eq:CDC3lemma2}, CDC4 is \eqref{eq:CDC4lemma}, and CDC5 is \eqref{eq:CDC5lemma}.

This leaves out a single axiom of CDCs, listed as CDC2 in~\cite{blute_cartesian_2009} and reported as~\eqref{eq:CDC2lemma} below, which requires $\delta f$ to be additive in its second component. In order to identify an analogue condition on $D$, we need a refinement of the simple fibration. It will amount to ensure that all maps in a fibre are additive in their second component by integrating the chosen monoids of a CLA $\cat B$ into the total category of a fibration of CLA categories. This is called \emph{the} additive bundle fibration in~\cite[\S 1.5]{blute_cartesian_2009}. Because it arises as a refinement of the simple fibration, we call it the simple CLA fibration. 
\begin{definition}[{\bf simple CLA fibration}]
\label{ex:add_cla_fib}
    Let $\cat B$ be a CLA category.
    Define $S_\CLA(\cat B)$ as having the same objects as $S(\cat B)$.
    Its maps are pairs ${f' \choose f} : {A' \choose A} \chartto {B' \choose B}$ where $f:A \to B$ is any map in $\cat B$ and $f':A \times A' \to B'$ additive in its second component, meaning for every $a:X \to A$, $v,w:X \to A'$, we have
    \begin{equation}
        f' \circ \langle a,h+k\rangle = \left( f' \circ \langle a,h \rangle \right) + \left( f' \circ \langle a,k \rangle \right) \qquad f' \circ \langle a,0\rangle = 0
    \end{equation}
    The \textbf{simple CLA fibration} $\simple_\CLA: S_\CLA(\cat B) \to \cat B$ over $\cat B$ is defined by projecting ${A' \choose A} \mapsto A$.
\end{definition}

CDC2, reported in~\eqref{eq:CDC2lemma} below, can now be explained in terms of sections by requiring that $D$ should not return maps in the total category of $S(\cat B)$ but rather maps in $S_\CLA(\cat B)$.

\begin{lemma}
    Imposing commutativity of diagrams $(\star)$ below for any $v,h,k \colon A \to B$ and $f \colon B \to C$ in $\cat B$
    \begin{equation}
       \begin{aligned}
       \vcenter{
    \xymatrix{
    & C &&& \\
   & B \times B \ar[u]^{\delta f} \ar@/^5pt/[drrr]^{\pi_1} \ar@/^15pt/[drrrr]^{\pi_0} &&& & \\
    & \ar@/_15pt/[rrrr]_{v} A \ar@{{}{ }{}}@/^22pt/[uu]|{(\star)}
     \ar@/^35pt/[uu]^{\left( \delta f \circ \langle v,h\rangle \right) + \left( \delta f \circ \langle v,k\rangle \right)} \ar@{.>}[u]_{\langle v, h+k\rangle} \ar[rrr]^{h+k} &&& B & B
    }
    }
    \\
           \vcenter{
    \xymatrix{
    & C &&& \\
   & B \times B \ar[u]^{\delta f} \ar@/^5pt/[drrr]^{\pi_1} \ar@/^15pt/[drrrr]^{\pi_0} &&& & \\
    & \ar@/_15pt/[rrrr]_{v} A \ar@{{}{ }{}}@/^22pt/[uu]|{(\star)}
     \ar@/^35pt/[uu]^{0_{A,C}} \ar@{.>}[u]_{\langle v, 0_{A,B}\rangle} \ar[rrr]^{0_{A,B}} &&& B & B
    }
    }
 \label{eq:CDC2lemma}
 \end{aligned}
    \end{equation}
     is equivalent to asking that $D$ restricts to a section of the simple CLA fibration $\simple_\CLA: S_\CLA(\cat B) \to \cat B$.
\end{lemma}
\begin{proof}
   The requirement on $\delta f$ imposed by the restriction to $S_\CLA(\cat B)$ amounts precisely to asking that is additive in its second component, i.e.~Equation~\eqref{eq:CDC2lemma}.
\end{proof}

We can now summarise our discussion via the following theorem, stating the equivalence between the first-order axioms of a CDC and the existence of a section with the desired properties.

\begin{therm}
\label[theorem]{th:FODS-CDCs}
    Let $\cat B$ be a CLA category.
    Let $D : \cat B \to S_\CLA(\cat B)$ be a stationary CLA section of the CLA simple fibration $\simple_\CLA: S_\CLA(\cat B) \to \cat B$:
    \begin{equation}\label{eq:thmCDC}
        \begin{tikzcd}[ampersand replacement=\&,sep=scriptsize]
        	{S_\CLA(\cat B)} \& {{A \choose A}} \& {{B \choose B}} \\
        	\\
        	{\cat B} \& A \& B
        	\arrow["{{\simple_\CLA}}"', from=1-1, to=3-1]
        	\arrow["D"', curve={height=15pt}, dashed, from=3-1, to=1-1]
        	\arrow[""{name=0, anchor=center, inner sep=0}, "f", from=3-2, to=3-3]
        	\arrow["\delta f", shift left, from=1-2, to=1-3]
        	\arrow["f"', shift right, from=1-2, to=1-3]
        	\arrow[""{name=1p, anchor=center, inner sep=0}, phantom, from=1-2, to=1-3, start anchor=center, end anchor=center]
        	\arrow[shorten <=17pt, shorten >=17pt, maps to, from=0, to=1p]
        \end{tikzcd}
    \end{equation}
    The existence of $D$ as in \eqref{eq:thmCDC} is equivalent to imposing a first-order CDC structure on $\cat B$, that is, the operator $\delta$ satisfies the following axioms    \begin{description}
        \item[CDC.1] $\delta$ preserves the left-additive structure on $\cat B$:
        \begin{equation}
        \label{eq:cdc1}
            \delta (f+g) = \delta f + \delta g \qquad \qquad \delta 0 = 0
        \end{equation}
        \item[CDC.2] Given $f \colon A\to B$, $\delta f \colon A \times A \to B$ is additive in its second component:
             \begin{equation}
         \label{eq:cdc2}
    	            \delta f \circ \langle a,h+k\rangle = \left( \delta f \circ \langle a,h \rangle \right) + \left( \delta f \circ \langle a,k \rangle \right) \qquad \delta f \circ \langle a,0\rangle = 0
        \end{equation}
        \item[CDC.3] $\delta$ preserves identities and projections: $\delta(\id_A) = \pi_1$, and if $\pi_A \colon A \times B \to A$ and $\pi_B \colon A \times B \to B$ are the projections, then for $\delta \pi_A \colon (A\times B) \times (A \times B) \to A$ it holds that
        \begin{equation}\label{eq:CDC3}
           \delta \pi_A = \pi_A \circ \pi_1 \qquad \qquad \delta\pi_B = \pi_B \circ \pi_1
        \end{equation}
        \item[CDC.4] $\delta$ preserves pairings:
        \begin{equation}
            \delta\langle f,g \rangle = \langle \delta f, \delta g \rangle
        \end{equation}
        \item[CDC.5] $\delta$ satisfies the chain rule: if $A \nto{f} B \nto{g} C$ are composable morphisms in $\cat B$, we have:
        \begin{equation}
           \delta (g \circ f) = \delta g \circ \langle \delta f, f \circ \pi_1  \rangle
        \end{equation}
    \end{description}
\end{therm}

\begin{remark} The abstract treatment of CDCs as sections of the simple fibration not only explains CDC axioms via fundamental categorical properties, such as functoriality and product preservation, but also suggests a reformulation of those axioms. Indeed, Lemmas~\ref{lemma:CDC1} and \ref{lemma:CDC2} indicate that~\eqref{eq:CDC3lemma} conceptually belongs to CDC5 rather than CDC3, as part of a functoriality requirement. Also, the remaining part of CDC3 should go with CDC4, to express a product preservation requirement.
\end{remark}

\subsection{Generalised Cartesian Differential Categories}
In the previous section we reformulated the first-order axioms of CDCs as properties of a stationary section $D$ of the simple CLA fibration.
Stationary meant the action of $D$ on an object is defined by $DA = {A \choose A}$. Thus implicit in this model of differentiation is that the type of the input of a function (first component) is the same as the type of the tangent vectors (second component). However, as argued in~\cite{cruttwell_cartesian_2017}, this is very rarely the case. \textbf{Generalised CDCs} (gCDCs) take various steps to fix this problem (see~\cite{cruttwell_cartesian_2017} for detailed motivation).

The first crucial different between gCDCs and CDCs is the absence of CLA structures for the first. The commutative monoids which give us the additive structure on spaces of tangent vectors are directly chosen by the derivative functor instead of being baked in the structure of the spaces themselves.

In order to characterise gCDCs, we need a slightly different recipe than with CDCs. We can no longer assume objects of the base category ${\cat B}$ have a chosen commutative monoid structure, but rather let the monoids come from the space where the tangent vectors live, i.e.~the total category $S(\cat B)$. The resulting notion is called additive simple fibration. 

\begin{definition}[{\bf Additive Simple Fibration}]
\label{ex:add_simple_fib}
Given a category ${\cat B}$ with finite products, define $S^+(\cat B)$ as the category whose objects are pairs ${B' \choose B}$, with $B$ an object of $\cat B$ and $B'$ a commutative monoid in the fibre $S(\cat B)_B$  (where $(S(\cat B)$ is defined in Example~\ref{ex:simple_fib}). We write
    \begin{equation}
        0^{B} : B \to B', \quad +^{B} : B \times B' \times B' \to B'.
    \end{equation}
    for the monoid structure.

      A map in $S^+(\cat B)$ is a map ${f' \choose f} : {A' \choose A} \chartto {B' \choose B}$ in $S(\cat B)$, except now $f'$ is asked to preserve the monoid structures; that is, the following diagrams commute:

    \begin{equation}
    \label{eq:fibre_linearity}
        \begin{tikzcd}[ampersand replacement=\&,sep=scriptsize]
            A \& {B} \& \&  A \times A' \times A' \& A \times A' \\
            {A \times A'} \& {B'} \& \&  B \times B' \times B' \& B' \\
            \arrow["{f'}", from=2-1, to=2-2]
            \arrow["f", from=1-1, to=1-2]
            \arrow[" {0^B}", from=1-2, to=2-2]
            \arrow[" {\<\id_A,0^A\>} "', from=1-1, to=2-1]
            \arrow["{+^{A}}", from=1-4, to=1-5]
            \arrow["{f'_2}", from=1-4, to=2-4]
            \arrow["{+^{B}}", from=2-4, to=2-5]
            \arrow["{f'}", from=1-5, to=2-5]
        \end{tikzcd}
    \end{equation}
where $f'_2 = \<f \circ \pi_0, f' \circ \<\pi_0, \pi_1\>, f' \circ \<\pi_0, \pi_2\>\>$.  The \textbf{additive simple fibration} $\simple^+ \colon S^+(\cat B) \to {\cat B}$ is defined as a functor like the simple fibration (Definition~\ref{ex:simple_fib}). 

\end{definition}

\begin{remark}
    The additive fibration simple considers all commutative monoids in the fibers of $S(\cat B)$. However, if $\cat B$ already had commutative monoids, by being a CLA, we could consider only those chosen ones.
    More formally, for a given CLA category $\cat B$, the fibration of CLA categories given in Definition~\ref{ex:add_cla_fib} is a subfibration of the additive simple fibration of Definition~\ref{ex:add_simple_fib}, namely the one spanned by picking in $S^+(\cat B)$ only the monoids that are chosen by the CLA structure on $\cat B$.
\end{remark}


Most of the structure of a generalised CDC can be captured very cleanly by specifying a section of the additive simple fibration. Apart from the `type-level' changes, generalised CDCs are pretty much the same as CDCs, so we do not linger on each axiom anymore and go straight to the formulation of the characterisation theorem. For the sake of uniformity with the notation of the previous subsection, we write $\lambda$ for the differential operator of gCDCs. Also, we use the same axiom names as in~\cite{cruttwell_cartesian_2017}.

\begin{therm}
\label[theorem]{th:FODS-gCDCs}
    Let $\cat B$ be a cartesian category.
    Let $L : \cat B \to S^+(\cat B)$ be a product-preserving section of the additive simple fibration:
    \begin{equation}\label{eq:thmgCDC}
        \begin{tikzcd}[ampersand replacement=\&,sep=scriptsize]
        	{S^+(\cat B)} \& {{\lambda A \choose A}} \& {{\lambda B \choose B}} \\
        	\\
        	{\cat B} \& A \& B
        	\arrow["{{\simple^+_{\cat B}}}"', from=1-1, to=3-1]
        	\arrow["L"', curve={height=15pt}, dashed, from=3-1, to=1-1]
        	\arrow[""{name=0, anchor=center, inner sep=0}, "f", from=3-2, to=3-3]
        	\arrow["\lambda f", shift left, from=1-2, to=1-3]
        	\arrow["f"', shift right, from=1-2, to=1-3]
        	\arrow[""{name=1p, anchor=center, inner sep=0}, phantom, from=1-2, to=1-3, start anchor=center, end anchor=center]
        	\arrow[shorten <=17pt, shorten >=17pt, maps to, from=0, to=1p]
        \end{tikzcd}
        \quad \text{where }\ \lambda f : A \times \lambda A \to \lambda B.
    \end{equation}
    The existence of $L$ as in~\eqref{eq:thmgCDC} is equivalent to imposing a first-order generalised CDC structure on $\cat B$, namely, for each object $A$, a commutative monoid $(\lambda A, +^{A}, 0^{A})$, and for each map $f: A \to B$, a map $\lambda f: A \times \lambda A \to \lambda B$, such that
    \begin{description}
        \item[gCDC.1] $\lambda$ preserves the commutative monoid structure on each monoid $\lambda X$:
        \begin{equation}
        \label{eq:cd1}
            \lambda(+^A) = +^{\lambda A} \circ \pi_1, \quad \lambda(0^A) = 0^{\lambda A} \circ \pi_1;
        \end{equation}
        \item[gCDC.2] for each map $f:A \to B$, $\lambda f : A \times LA \to LB$ is a map of monoids in the second component (as in Equation \ref{eq:fibre_linearity})
        \item[gCDC.3] $\lambda$ preserves identities and projections: $\lambda(\id_A) = \pi_1$, and if $\pi_A :A \times B \to A$ is a projection, we have
        \begin{equation}
        \label{eq:cd3}
            \lambda(\pi_A) = \pi_A \circ \pi_1
        \end{equation}
        (and similarly for $\pi_B$) and in case $A=1$, this also means that $\lambda(!_B) = \pi_1$,
        \item[gCDC.4] $\lambda$ preserves pairings:
        \begin{equation}
            \lambda\langle f,g \rangle = \langle \lambda f, \lambda g \rangle,
        \end{equation}
        \item[gCDC.5] $\lambda$ satisfies the chain rule: if $A \nto{f} B \nto{g} C$ are composable morphisms in $\cat B$, we have:
        \begin{equation}\label{eqn:gcdc_composition}
        \lambda(g \circ f) = \lambda(g) \circ \<f \circ \pi_0, \lambda(f)\>
        \end{equation}
    \end{description}
\end{therm}

\begin{proof}
    Axiom \textbf{gCDC.2} holds by definition of $S^+(\cat B)$.
    \textbf{gCDC.1}, the projection-preservation part of \textbf{gCDC.3}, and \textbf{gCDC.4} hold because $L$ is product-preserving. Specifically, the right hand sides of~\eqref{eq:cd1} are the commutative monoid structures on $LA$ presented as maps in $S^+(\cat B)$, and the right hand side of~\eqref{eq:cd3} is what a product projection looks like in the same category.
    Finally, identity preservation and \textbf{gCDC.5} holds by functoriality of $L$ and because the right hand side of \ref{eqn:gcdc_composition} is the composition rule of $S^+(\cat B)$.
\end{proof}

As with the other differential structures considered, our characterisation leaves out  the higher-order aspects. In the case of generalised CDCs, this includes the axiom ``$\lambda (\lambda A) = \lambda A$'', which does not even typecheck  when we see $L$ as a section.

\subsection{Reverse Derivative Categories}
\label{sec:rdcs}
In machine learning (see~\cite{griewank2012invented}), `reverse derivatives' are what is known in geometry as `pullback of differential 1-forms', which is the action on smooth maps of the cotangent functor.
Functoriality of the cotangent bundle construction is not as well-known as that of tangent bundles because it is hard to state it without adopting a more abstract point of view.
The fibrational point of view does exactly this, highlighting how reverse derivatives arise via sections of the dual fibration.
We will see this in the case of \emph{reverse cartesian differential  categories}, which were introduced in~\cite{cockett_reverse_2020} as a reverse counterpart of CDCs. We show their first-order structure is exactly a section of the dual fibration of the simple CLA fibration which we call the lens CLA fibration. We start by recalling the lens fibration $\lens_{\cat B} : L(\cat B) \to \cat B$ of Section~\ref{sec:fibs} which, as remarked there, arises as the dual of the simple fibration, that is, $\lens_{\cat B} = \simple_{\cat B}^{\vee}$. In order to take account of the CLA-structure, we mirror the treatment of CDCs, with the following definition.

\begin{definition}[{\bf Lens CLA fibration}]
    Let $\cat B$ be a CLA category.
    The \textbf{lens CLA fibration} $\lens_{CLA}(\cat B)$ is the dual of the simple CLA fibration $\simple_{\CLA}(\cat B)$.
   That is, define $L_{\CLA}(\cat B) := S_{\CLA}(\cat B)^{\vee}$ and
    \begin{equation}
        \lens_{CLA}(\cat B) := \simple_{\CLA}(\cat B)^{\vee} \colon L_{\CLA}(\cat B) \to \cat B
    \end{equation}
    Concretely, an object in $L_{\CLA}(\cat B)$ is a pair ${A' \choose A}$, and a map ${A' \choose A} \to {B' \choose B}$ is a pair ${f^\sharp \choose f}$ where $f: A \to B$ and $f^\sharp: A \times B' \to A'$ is additive in its second component.
\end{definition}

For this definition to be useful, we need the additive lens fibration to be a CLA fibration with fibred biproducts. This we prove next.

\begin{lemma}
\label[lemma]{lemma:dual-fibred-biprod}
    Let $\cat B$ be a cartesian category.
    If $q:\cat E \to \cat B$ is a fibration of cartesian categories with fibred biproducts, then so is $q^{\vee}:\cat E^{\vee} \to \cat B$.
\end{lemma}
\begin{proof}
    Each fibre of $q^\vee$ has fibred biproducts since the opposite of a category with biproducts also has biproducts.
    Since $\cat B$ has products, we can conclude by~\cref{lem:fibprod-cartfib} that $q^\vee$ is a fibration of cartesian categories.
\end{proof}

\begin{lemma}
\label[lemma]{lemma:dual-of-CLA fib}
    If $q:\cat E \to \cat B$ is a fibration of CLA categories, then so is $q^{\vee}:\cat E^{\vee} \to \cat B$.
\end{lemma}
\begin{proof}The objects of $\cat E$ are those of $\cat E^\vee$ and so the chosen monoids are the same, and so the result follows.
\end{proof}

\begin{corollary}\label{cor:lens_cla}
$\lens_{\CLA}(\cat B) \colon L_{\CLA}(\cat B) \to \cat B$ is a fibration of CLA categories.
\end{corollary}
\begin{proof}
This immediately follows from the previous two results.  However, in what comes next, it will be useful to explicitly spell out the product structure.  So, suppose we have objects ${A' \choose A}$ and ${B' \choose B}$; their product object is simply ${A' \times B' \choose A \times B}$.  The projections are given by
	\[ {\<\pi_{A'}, 0\> \choose \pi_A}: {A' \times B' \choose A \times B} \opticto {A' \choose A} \mbox{ and } {\<0, \pi_{B'} \> \choose \pi_B}: {A' \times B' \choose A \times B} \opticto {B' \choose B}. \]
Given another object ${X' \choose X}$ with maps
	\[ {f^\sharp \choose f}: {X' \choose X} \opticto {A' \choose A} \mbox{ and } {g^\sharp \choose g}: {X' \choose X} \opticto {B' \choose B} \]
then the pairing of them uses sum:
	\[ {f^\sharp \circ \<\pi_X, \pi_{A'}\> + g^\sharp \circ \<\pi_X, \pi_{B'}\> \choose \<f,g\>}  : {X' \choose X} \opticto {A' \times B' \choose A \times B} \]
\end{proof}

We can now give a fibrational presentation of the first order axioms of reverse cartesian differential categories, which we name RDC1-5 following~\cite{cockett_reverse_2020}. This result was anticipated by~\cite[Proposition~31]{cockett_reverse_2020} and~\cite{cruttwell_categorical_2022}.
We give here a complete and clear proof, though it basically amounts to that of~\cref{th:FODS-CDCs}. In the statement, we use a stationary section (see Section~\ref{sec:CDCsections}): recall that it maps $A$ to ${A \choose A}$. Also, we write $\rho$ for the differential operator, mirroring the notation used for CDCs and gCDCs.

\begin{therm}
\label[theorem]{th:FODS-RDCs}
    Let $\cat B$ be a CLA category.
    Let $R : \cat B \to L_{\CLA}(\cat B)$ be a stationary CLA section of the additive lens fibration:
    \begin{equation}
        \begin{tikzcd}[ampersand replacement=\&,sep=scriptsize]
        	{L_{\CLA}(\cat B)} \& {{A \choose A}} \& {{B \choose B}} \\
        	\\
        	{\cat B} \& A \& B
        	\arrow["{\lens_{\CLA}}"', from=1-1, to=3-1]
        	\arrow["R"', curve={height=15pt}, dashed, from=3-1, to=1-1]
        	\arrow[""{name=0, anchor=center, inner sep=0}, "f", from=3-2, to=3-3]
        	\arrow["\rho f"', shift right, from=1-3, to=1-2]
        	\arrow["f"', shift right, from=1-2, to=1-3]
        	\arrow[""{name=1p, anchor=center, inner sep=0}, phantom, from=1-2, to=1-3, start anchor=center, end anchor=center]
        	\arrow[shorten <=17pt, shorten >=17pt, maps to, from=0, to=1p]
        \end{tikzcd}
        \quad \text{where}\ \rho f : A \times B \to A.
    \end{equation}
    Then the operation $\rho$ captures all the first-order aspects of an RDC structure on $\cat B$, namely:
    \begin{description}
        \item[RDC.1] $\rho$ preserves the left-additive structure on $\cat B$:
        \begin{equation}
        \label{eq:rdc1}
            \rho(0) = 0, \quad \rho(f+g) = \rho f + \rho g,
        \end{equation}
        \item[RDC.2] for each map $f:A \to B$, $\rho f : A \times B \to A$ is additive in its second component:
         \begin{equation}
         \label{eq:cdc1}
    	            \rho f \circ \langle a,h+k\rangle = \left( \rho f \circ \langle a,h \rangle \right) + \left( \rho f \circ \langle a,k \rangle \right) \qquad \rho f \circ \langle a,0\rangle = 0
        \end{equation}
        \item[RDC.3] $\rho$ preserves identities and projections: $\rho(\id_A) = \pi_1$, and if $\pi_A :A \times B \to A$ is a projection, we have
        \begin{equation}
        \label{eq:rdc3}
            \rho(\pi_A) = \<\pi_2, 0\>
        \end{equation}
        (and similarly for $\pi_B$) and in case $A=1$, this also means that $\rho(!_B) = 0$,
        \item[RDC.4] $\rho$ preserves pairings:
        \begin{equation}
            \rho \langle f,g \rangle = (\id \times \pi_1)Rf + (\id \times \pi_2)Rg,
        \end{equation}
        \item[RDC.5] $\rho$ satisfies the (reverse) chain rule: if $A \nto{f} B \nto{g} C$ are composable morphisms in $\cat B$, we have:
        \begin{equation}
        \label{eq:rdc5}
            \rho(g \circ f) = \rho(f) \circ \<\pi_0, \rho(g) \circ \<f \circ \pi_0, \pi_1\>\>
        \end{equation}
    \end{description}
\end{therm}

\begin{proof}
    \textbf{RDC.1} holds because $R$ preserves left-additive structure in quality of CLA functor.  The projection-preservation part of \textbf{RDC.3} and \textbf{RDC.4} hold since $R$ is product-preserving (projections and pairing in $S^+(\cat B)^\vee$ were described above in Corollary \ref{cor:lens_cla}).
    Finally, \textbf{RDC.2}, the identity preservation part of \textbf{RDC.3}, and \textbf{RDC.5} hold since that's the structure of $\lens_{CLA}(\cat B)$: vertical maps are additive and compose as in~\eqref{eq:rdc5}.
\end{proof}

One of the key results of \cite{cockett_reverse_2020} was that every RDC gives a CDC (see Theorem 16).  This was established directly (that is, by defining a forward derivative $D$ from a reverse derivative $R$, and checking the axioms one-by-one).  However, this direct approach requires fairly lengthy computations; moreover, it does not provide much insight into where the result comes from or whether it may be true in more general settings.    However, the fibrational approach we have outlined here gives an alternative way to understand this result.  In particular, by thinking of CDCs and RDCs as sections, the next two results allow one to define the CDCs from an RDC via a composite of functors.

We begin with a result that is interesting in its own right: a natural functor from lenses of lenses to the simple fibration.  (However, note that to even state this result requires that we work with \emph{CLA} lenses).

\begin{lemma}\label{lemma:phi}
There is a CLA functor
	\[ \Phi: L_{CLA}(L_{CLA}(\cat B)) \to S_\CLA(\cat B). \]
\end{lemma}
\begin{proof}
The objects of $L_{CLA}(L_{CLA}(\cat B))$ are quadruples
	\[ {{B' \choose B} \choose {A' \choose A}} \]
and we define $\Phi$ of such an object to be the outer components: ${B' \choose A}$.

A map in $L_{CLA}(L_{CLA}(\cat B))$
	\[ {{B' \choose B} \choose {A' \choose A}} \to {{D' \choose D} \choose {C' \choose C}} \]
then consists of lens maps
	\[ {f^\sharp \choose f}: {A' \choose A} \to {C' \choose C} \mbox{ and } {g^\sharp \choose g}: {A' \choose A} \times {D' \choose D} \to {B' \choose B}. \]
In particular, $f: A \to C$, and $g^\sharp: A \times D \times B' \to A' \times D'$ (recall from Corollary \ref{cor:lens_cla} that ${A' \choose A} \times {D' \choose D} = {A' \times D' \choose A \times D}$).

Then $\Phi$ of such a map should be a simple fibration map from ${B' \choose A}$ to ${D' \choose C}$, so must consist of maps
	\[ A \to C \mbox{ and } A \times B' \to D'. \]
Then we can simply take $\Phi$ of the above to have first map $f$, and second map the composite
\begin{equation*}
     \begin{tikzcd}[ampersand replacement=\&, sep=scriptsize]
            {A \times B'} \& \& A \times D \times B' \& {A' \times D'} \& D'
            \arrow["{\<\pi_A, 0, \pi_{B'}\>}", from=1-1, to=1-3]
            \arrow["{g^*}", from=1-3, to=1-4]
            \arrow["{\pi_{D'}}", from=1-4, to=1-5]
        \end{tikzcd}
\end{equation*}
It is then straightforward to check that this defines a CLA functor.
\end{proof}

\begin{therm}
    \label[theorem]{thm:cdc_from_rdc}
    Let $\cat B$ be a CLA and $R : \cat B \to L_{\CLA}(\cat B)$ a stationary CLA section of the additive lens fibration. Then we can build an associated stationary CLA section of the additive simple fibration.  (In other words, every first-order RDC gives a first-order CDC).
\end{therm}

\begin{proof}
A CLA section $R:{\cat B} \to L_{\CLA}(\cat B)$ can be lifted to
\[ L_{\CLA}(R): L_{\CLA}(\cat B) \to L_{\CLA}L_{\CLA}(\cat B). \]
We can then get the desired CLA section of  $S_\CLA(\cat B)$ by the triple composite
\begin{equation*}
     \begin{tikzcd}[ampersand replacement=\&, sep=scriptsize]
            {\cat B} \& {L_{\CLA}(\cat B)} \& \& {L_{\CLA}L_{\CLA}(\cat B)} \& {S_\CLA(\cat B)}
            \arrow["{R}", from=1-1, to=1-2]
            \arrow["{L_{\CLA}(R)}", from=1-2, to=1-4]
            \arrow["{\Phi}", from=1-4, to=1-5]
        \end{tikzcd}
\end{equation*}
where $\Phi$ is defined in Lemma~\ref{lemma:phi}.
\end{proof}
It is straightforward to check that the resulting structure is the same as in \cite[Theorem 16]{cockett_reverse_2020}: for a map $f: A \to B$, a reverse derivative $R$ gives a forward derivative $D$ by the composite
	\[ D[f] := \pi_B \circ R[R[f]] \circ \<\pi_A, 0, \pi_B\>. \]
However, the above result allows one to have a higher-level understanding of how this comes about, and how it could potentially be generalised (for more on this, see~\cref{sec:reverse_tan}).

\subsection{Tangent Categories}
We now go back to `forward' derivative structures.
Moving further up the generality ladder from generalised CDCs, we get to tangent categories, a concept introduced in~\cite{Rosicky84} and revamped and further developed in~\cite{cockett_differential_2014}. A tangent category is a category equipped with a ``tangent functor'' $\tau: \cat B \to \cat B$ with properties mimicking those of the tangent bundle construction on smooth manifolds.  Each CDC is a tangent category (with $\tau(A) := A \times A$) and each gCDC is a tangent category (with $\tau(A) := A \times M(A)$).  However, for a general smooth manifold $A$, its tangent bundle $\tau(A)$ need not be of the form $A \times (-)$.  Tangent categories thus reflect this greater generality, merely asking each object $A$ have a ``tangent bundle'' $\tau(A)$ over it (satisfying various categorical axioms).

At first glance, it is not hard to see how the definition of a tangent category might be formulated fibrationally.
First, the tangent functor comes with a projection natural transformation $p:\tau \to 1$, ie for every $A$, there is a map $p_A:\tau A \to A$. Thus $\tau$, when equipped with $p$, defines a section of the codomain fibration over $\cat B$ which maps
$A$ to $p_A$. Furthermore, (i) the tangent category definition asks that for each $n$, the pullback of $n$ copies of $p_A:\tau A \to A$ exists; this can be stated more abstractly as requiring $n$-fold products of $p_A$ in the fibre $\cat B/A$, and (ii) $p_A$ is required to possess the structure of a commutative monoid in the fibre $\cat B/A$. Thus, one might conjecture the first-order aspects of a tangent category to be equivalent to a section of the fibration of \emph{additive bundles}, which maps an object in the base of a fibration to a commutative monoid in the fibre above that object (see below).

There is a slight problem in the story above. The definition of a tangent category requires very few pullbacks to exist. Since pullbacks are reindexing in the codomain fibration, this means that to model the exact definition of a tangent category, fibrations may be too strong a requirement. However, as the theory of tangent categories has been developed, it has often been very helpful to add the assumption that all pullbacks along differential bundles exist.  For example, when defining a connection on a differential bundle $q: E \to A$, we need to have the object which is the pullback of $q$ along $p_A: \tau A \to A$ \cite[Definition 4.5]{connections}. We could of course still take a minimal approach which adds appropriate pullbacks only as and when required, but that leads to a fragmented theory, where the precise set of pullbacks present at any given time is hard to keep track of.
So, instead, we take the approach articulated in the introduction: we do not seek to replicate the exact definitions present in the literature, but rather identify fibrational variants which have good behaviour, and which can be used in specific settings if they offer added value. Of course, we do not want to stray too far from the existing body of work: for example, the key example of smooth manifolds does not have all pullbacks, and therefore we do not want to build a model that excludes this key example. This motivates our focus on categories of bundles and existence of pullbacks of bundles.

In order to execute the above programme, we need to define what a fibration of additive bundles is. Just as the canonical construction of a category with biproducts from one with products is the category of internal commutative monoids, we can build a fibration with fibred biproducts from one with fibred products (Definition~\ref{def:fibredproducts}) by taking commutative monoids in each fibre of the fibration:

\begin{definition}[{\bf Additive Fibration}]
    Let $q:\cat E \to \cat B$ be a fibration with fibred products. Let $\cat E^+$ be the category with objects pairs $(B,M)$, where
    $B$ is an object of $\cat B$ and $M$ is a commutative monoid in the fibre $\cat E_B$. A morphism of $\cat E^+$ from $(B,M)$ to $(B',M')$ consists of a map $f:B \to B'$ and commutative monoid homomorphism $f':M \to f^* M'$, where $f^*M'$ is the commutative monoid arising from reindexing $M'$ along $f$. We define the  \textbf{cofree fibration with fibred biproducts} (or simply \textbf{additive fibration}) over $q$ as the map $q^+ = qU : \cat E^+ \to \cat B$ given as follows (where $U$ is the natural forgetful functor).
    \begin{equation}
        \begin{tikzcd}[ampersand replacement=\&, sep=scriptsize]
            {\cat E^+} \&\& {\cat E} \\
            \& {\cat B}
            \arrow["q^+"', from=1-1, to=2-2]
            \arrow["q", from=1-3, to=2-2]
            \arrow["U", from=1-1, to=1-3]
        \end{tikzcd}
    \end{equation}
\end{definition}
Note that by~\cref{lemma:cofree-semiadd} the fibration $q^+$ defined as above indeed has fibred biproducts. For our purposes here, the main case is the additive fibration over a fibration of bundles (\emph{cf.}~\cref{ex:bun_fib}). Given a fibration of bundles $\cod_{\cat B} \colon {\cat B^{\downbundle}} \to {\cat B}$, we refer to the additive fibration $\cod_{\cat B}^+ \colon {\cat B^{\downbundle+}} \to {\cat B}$ over $\cod_{\cat B}$ as the \textbf{additive bundle fibration}.

\begin{therm}
\label[theorem]{th:FODS-TCs}
    Let $\cat B$ be a category with bundles $\cat F \subseteq \cat B$ and let $T \colon \cat B \to \cat B^{\downbundle+}$ be a functor of categories with bundles section of the additive bundle fibration:
    \begin{equation}
        \begin{tikzcd}[ampersand replacement=\&,sep=scriptsize]
            {\cat B^{\downbundle+}} \& {\begin{matrix}TA\\\downbundle\\A\end{matrix}} \& {\begin{matrix}TB\\\downbundle\\B\end{matrix}} \\
            \\
            {\cat B} \& A \& B
            \arrow["{\cod_{\cat B}^+}"', from=1-1, to=3-1]
            \arrow["T"', curve={height=15pt}, dashed, from=3-1, to=1-1]
            \arrow[""{name=0, anchor=center, inner sep=0}, "f", from=3-2, to=3-3]
            \arrow["T f", shift left=6.5, from=1-2, to=1-3]
            \arrow["f", shift right=7, from=1-2, to=1-3]
            \arrow[""{name=1p, anchor=center, inner sep=0}, phantom, from=1-2, to=1-3, start anchor=center, end anchor=center]
            \arrow[shorten <=16pt, shorten >=30pt, maps to, from=0, to=1p]
        \end{tikzcd}
    \end{equation}
    Then assuming the existence of such $T$ is equivalent to assuming all the first-order aspects of a tangent category structure on $\cat B$, namely:
    \begin{enumerate}
        \item There is a copointed endofunctor $\tau:\cat B \to \cat B$ with natural transformation $p:\tau \twoto \id_{\cat B}$.
        \item All finite pullback powers of components of $p$ exist (that is, for each $n$ and each $A$ the pullback of $n$ copies of $p_A: \tau(A) \to A$ along itself exists) and $\tau$ preserves these pullbacks.
        \item There are natural transformations $0_B : B \to \tau B$ and $+_B : \tau B \times_B \tau B \to \tau B$ making each $p_B:\tau B \to B$ an additive bundle, that is, a commutative monoid in the slice  category $\cat B/B$.
    \end{enumerate}
\end{therm}
\begin{proof}
     Given $T : \cat B \to \cat B^{\downbundle +}$, we get an endofunctor by composing $T$ with the domain functor $\dom^+ : \cat B^{\downbundle +} \to \cat B$
    \begin{equation}
        \tau := \cat B \nlongto{T} \cat B^{\downbundle +} \nlongto{\dom^+_{\cat B}} \cat B.
    \end{equation}
    This is copointed, by whiskering the obvious natural transformation $\dom^+_{\cat B} \twoto \cod^+_{\cat B}$ with $T$:
    \begin{equation}
        p :=
        \begin{tikzcd}[ampersand replacement=\&,sep=scriptsize]
            \& {\cat B^{\downbundle+}} \\
            {\cat B} \&\&\& {\cat B}
            \arrow["T", curve={height=-6pt}, from=2-1, to=1-2]
            \arrow[""{name=0, anchor=center, inner sep=0}, "{\cod^+_{\cat B}}"'{pos=0.3}, shift left=2, curve={height=6pt}, from=1-2, to=2-4]
            \arrow["{\id_{\cat B}}"', Rightarrow, no head, from=2-1, curve={height=15pt}, to=2-4]
            \arrow[""{name=1, anchor=center, inner sep=0}, "{\dom^+_{\cat B}}", curve={height=-12pt}, from=1-2, to=2-4]
            \arrow[shorten <=2pt, shorten >=2pt, Rightarrow, from=1, to=0]
        \end{tikzcd}
    \end{equation}
    Indeed, $T \circ \cod^+_{\cat B} = \id_{\cat B}$ because $T$ is a section of $\cod^+_{\cat B}$. Observe that, as a result of this definition of $p$, $T$ maps $B$ to the bundle $p_B : TB \to B$ in the definition of a tangent category. Since $p_B$ is a bundle, it can be pulled back along itself as many times as one wishes, and since $T$ is a functor of categories with bundles, $T$ preserves such pullbacks. By virtue of being a right adjoint, $\dom^+$ preserve limits; hence $\tau$ preserve pullback powers of components of $p$. Finally, the natural transformations $0$ and $+$ are induced by projecting with $\dom^+$ the commutative monoid structure that each $p_B :TB \to B$ has in $\cat B^{\downbundle +}$.

    Conversely, given a copointed endofunctor $\tau$, we construct a section of the additive codomain fibration. The components of the copoint of $\tau$, call it $p$, organise into a section $T:\cat B \to \cat B^{\downbundle +}$ sending $B:\cat B$ to $p_B$, and a map $f:B \to B'$ to the corresponding naturality square for $p$.
    By virtue of being bundles and additive this is well-defined.
    Moreover, it is easy to check that $T$ is a functor of categories with bundles.
\end{proof}

\subsection{Reverse tangent categories}\label{sec:reverse_tan}

As noted in the previous section, a CDC cannot adequately capture the derivative on spaces like smooth manifolds: in a CDC, the (forward) derivative of a map $A \to B$ must be of type $A \times A \to B$, but for general smooth manifolds, the (forward) derivative of a map $A \to B$ has type $TA \to TB$, where $TA$ need not even be of the form $A \times (-)$.

A similar issue happens when one attempts to work with reverse derivatives on smooth manifolds.  A RDC gives for each map $A \to B$ a reverse derivative of type $A \times B \to A$, but again this is inadequate for smooth manifolds.  For smooth manifolds, instead one has a map $A \times_{B} T^*B \to T^*A$, where $T^*$ is the \emph{cotangent} bundle.  Categorically, this type of derivative is described by \emph{reverse tangent categories} (\cite{reverse_tangent}).  In this section, we will briefly recall this idea and sketch how the fibrational viewpoint can give new insight into this structure.

To begin, it is useful to recall different ways RDCs can be defined.  As discussed in Section \ref{sec:rdcs}, an RDC is a CLA $\cat B$ which has, for each map $f: A \to B$, a map $R[f]: A \times B \to A$.  Moreover, as per Theorem \ref{th:FODS-RDCs} the first-order aspect of this structure can be captured by a section of the dual of the simple CLA fibration, $\lens_{CLA}(\cat B)$.

However, there is an alternative characterization of RDCs.  As also noted previously, every RDC gives a CDC, and in fact one can precisely identify what extra structure on a CDC on $\cat B$ is needed to a build an RDC: a ``contextual linear dagger''.  This is essentially a functor from the simple CLA fibration of $\cat B$ to its dual.

It is this alternative characterization which the authors of  (\cite{reverse_tangent}) use to \emph{define} reverse tangent categories.  As they note in Remark 26 of that paper, ``we have found it difficult to provide a direct description of a reverse tangent category'' (a direct definition being one which axiomatizes a reverse derivative directly).  So, instead, taking inspiration from this alternative description of RDCs, they define a reverse tangent category to consist of a tangent category together with a ``linear involution'', which is a functor from a fibration of additive bundles (``differential'' bundles) to its dual (for more details, see \cite[Definition 23]{reverse_tangent}).

Taking the fibrational point of view, however, one can at least easily give a ``first-order'' theory of reverse tangent categories.  For this, let us use $\dlens_{A}(\cat B)$ (``additive dependent lenses'') to denote the dual of the additive bundle fibration over $\cat B$, that is,
	\[ \dlens_{A}(\cat B) := ({\cat B^{\downbundle+}})^{\vee} \]

\begin{definition}
\label[definition]{def:revtangentcat}
    A \textbf{first-order reverse tangent category} consists of a category $\cat B$ together with a section of the dual of the fibration of additive bundles,
	\[ R: \cat B \to \dlens_{A}(\cat B) \]
\end{definition}

It is then immediate that any reverse tangent category is a first-order reverse tangent category.

Now, part of the difficulty of trying to define a reverse tangent category in this way was that it was not clear to the authors of (\cite{reverse_tangent}) how to define a tangent category from such a structure.  However, Section \ref{sec:rdcs} offers an idea of how to do this.  If one could build a functor
	\[ \psi: \dlens_{A}(\dlens_{A}(\cat B)) \to  {\cat B^{\downbundle+}} \]
then given a first-order reverse tangent category $R: \cat B \to \dlens_{A}(\cat B)$, just as in Theorem \ref{thm:cdc_from_rdc}, one could build a first-order tangent category by the composite
\begin{equation*}
     \begin{tikzcd}[ampersand replacement=\&, sep=scriptsize]
            {\cat B} \& {\dlens_{A}(\cat B)} \& \& {\dlens_{A}(\dlens_{A}(\cat B))} \& {\cat B^{\downbundle+}}
            \arrow["{R}", from=1-1, to=1-2]
            \arrow["{\dlens_A(R)}", from=1-2, to=1-4]
            \arrow["{\psi}", from=1-4, to=1-5]
        \end{tikzcd}
\end{equation*}
Thus, this gives an idea of how to solve the ``direct definition'' problem of reverse tangent categories.  We leave a more detailed investigation of this idea to future work.

\subsection{The General Case: First-Order Differential Structures}
\label{sec:general-fods}
Our examples showcase a general pattern: given a certain structured category $\cat B$ (e.g.~with products, CLA, gCLA, bundles), the first-order aspects of a differential structure on that category are captured by the section of a specific kind of fibration with fibred biproducts over it.
This is what we call a \emph{first-order differential structure}, or \emph{FODS} for brevity.
The structure captured by a FODS is i) functoriality, which implies a form of the chain rule; ii) product preservation, which is a form of the product rule; and iii) additivity of the derivative over its tangent space. The rest of the structure pertains to the particular model of differentiation being discussed, and thus the definition of FODS is transparent to it. We now make this notion rigorous.

\begin{definition}[\bf FODS]
\label[definition]{def:fods}
    Let $\dblcat K$ be a 2-category with finite 2-limits, and let $\cat B$ be an object in it.
    A \textbf{first-order differential structure} (\textbf{FODS}) on $\cat B$ is a pair of a fibration $q:\cat E \to \cat B$ in $\dblcat K$ with fibred biproducts and a section $T$ thereof. A FODS is \textbf{cartesian} if $\cat B$, $\cat E$ have products and both $q$ and $T$ preserves them.
\end{definition}

Note that, in the above, the notion of $q$ in $\dblcat K$ having fibred biproducts generalises the analogous notion introduced in Definition~\ref{def:fibredproducts} for the case $\dblcat K=\Cat$.
For the sake of clarity, we may unpack it, as follows. First, a \textbf{cartesian} (resp. \textbf{cocartesian}) object in a 2-category with finite limits is an object $\cat B$ whose diagonal 1-cell $\Delta_{\cat B}:\cat B \to \cat B \times \cat B$ and terminal 1-cell $!_{\cat B} : \cat B \to 1$ admit right adjoints $\Delta_{\cat B} \adj \times$, $!_{\cat B} \adj \id_{\cat B}$ (resp. left adjoints $+ \adj \Delta_{\cat B}$, $0_{\cat B} \adj {!_{\cat B}}$). Thus an object is \textbf{semi-additive} if and only if $\Delta_{\cat B}$ and $!_{\cat B}$ admit ambidextrous adjoints, i.e.~both left and right adjoints exists and coincide. Now to say that a fibration $q:\cat E \to \cat B$ in $\dblcat K$ \textbf{has fibred biproducts} means that $q$ is a semi-additive object in $\Fib_{\dblcat K}(\cat B)$, i.e.~$q$ admits a fibred ambidextrous adjoint $\oplus : q \times_{\cat B} q \to q$, giving biproducts on each fibre.

A cartesian FODS in $\dblcat K$ is really a FODS in $\Cart{\dblcat K}$, i.e.~the 2-category of cartesian objects in $\dblcat K$. In fact, all our examples are cartesian except tangent categories, which have a commonly occurring cartesian variant, but are not cartesian from the start.
Other non-cartesian examples, which we do not discuss because they are slight variations of others we treat here, are monoidal forward and reverse differential categories, see~\cite{cruttwell_monoidal_2022}.

We summarise in Table~\ref{fig:FODS} how the examples covered in the previous sections fit the proposed pattern.

\begin{table}[h]
    \centering
    \begin{tabular}{c|c|c|c}
         & \textbf{2-category} & \textbf{fibration} & \textbf{extra requirements}\\
         \textbf{gCDCs} & $\CartMonCat$ & \text{additive simple} & \\
         \textbf{CDCs} & $\CLACat$ & \text{simple CLA} & \text{stationary section}\\
         \textbf{RDCs} & $\CLACat$ & \text{lens CLA} & \text{stationary section}\\
         \textbf{tangent categories} & $\BundleCat$ & \text{additive bundles} \\
         \textbf{reverse tangent categories} & $\BundleCat$ & \text{additive dependent lenses} &
    \end{tabular}
    \caption{first-order differential structures}\label{fig:FODS}
\end{table}


\begin{example}
    As an easy example, the first-order part of a \emph{cartesian} tangent category~\cite[Definition~2.8]{cockett_differential_2014} is a cartesian FODS, and it is in fact exactly the cartesian version of a first-order tangent category (\cref{th:FODS-TCs}).
\end{example}



\begin{example}
    \cref{def:fods} can also be used to guide new definitions.
    For instance, suppose we want to study a notion of `gCDC with coproducts'.
    If its first-order part is to be captured by a FODS respecting the structure, then it means we are working in the 2-category $\dblcat{Cocart}\Cat$ of cocartesian categories.
    This then disqualifies the additive simple fibration as a carrier of a gCDC with coproducts, since if $\cat B$ has coproducts, in general, $S^+(\cat B)$ does not have coproducts globally.
    Instead, extensive categories support both pullbacks and coproducts and thus their codomain fibration has global coproducts preserved by the fibration.
\end{example}

\section{Linearity and differential objects in FODS}
\label{sec:diff-objs}

For many categories with differential structure, spaces satisfying $TA \cong A \times A$ (where $T$ is our section) are special.
Examples are cartesian spaces in $\Smooth$, $R$-modules in affine schemes over a commutative ring $R$, convenient vector spaces in the category of convenient manifolds, and many others (see~\cite{cruttwell_differential_2023} for a recent survey).
These objects are called `differential'~\cite{cockett_differential_2014}. Here we show the first-order essence of a differential object and prove they are commutative monoids whose structure maps are linear with respect to a given differential structure.
This will be our~\cref{th:cmon-lin-obj-is-diff-obj}. We'll use this result to show that all FODS with finite products yield a cofree CDC when restricted to these `first-order' differential objects, generalizing the analogous result for tangent categories~\cite[Theorem~4.11, Theorem~4.12]{cockett_differential_2014}.

\subsection{Linear objects and maps}
\label{sec:linearity}
In this section, we let $\cat B$ be a category with finite products and $T: \cat B \to \cat B^\downarrow$ be a section of the additive codomain functor of $\cat B$ (no need for this to be a \emph{fibration} just yet).

\begin{definition}[\bf Linearity]
\label[definition]{def:linearity}
    Let $\cat B$ be a category with finite products, and let $T: \cat B \to \cat B^\downarrow$ be a section of the additive codomain functor of $\cat B$.
    A \textbf{$T$-linear object} is an object $B:\cat B$ equipped with a vertical isomorphism $ t  : p_B \cong \pi_1$ called \textbf{trivialization}:
    \begin{equation}
    \label{eq:lin-obj}
        \vspace*{5ex}
        \begin{tikzcd}[ampersand replacement=\&,sep=scriptsize]
            TB \&\& {B \times B} \\
            \& B
            \arrow["{\pi_1}", from=1-3, to=2-2]
            \arrow["\Overset{ t }{\cong}"{description}, from=1-1, to=1-3]
            \arrow["{p_B}"', from=1-1, to=2-2]
        \end{tikzcd}
        \vspace*{-5ex}
    \end{equation}
    Likewise, a \textbf{$T$-linear map} is a map $f:B \to B'$ for which the following commutes:
    \begin{equation}
    \label{eq:lin-map}
        \begin{tikzcd}[ampersand replacement=\&,sep=scriptsize]
            TB \& {TB'} \\
            {B \times B} \& {B' \times B'}
            \arrow["{f \times f}", from=2-1, to=2-2]
            \arrow["Tf", from=1-1, to=1-2]
            \arrow[" t ", from=1-2, to=2-2]
            \arrow[" t "', from=1-1, to=2-1]
        \end{tikzcd}
    \end{equation}
\end{definition}

Observe that $(-)^2 : \cat B \to \cat B^\downarrow$, which sends objects $B$ to the bundle $\pi_1 : B \times B \to B$ and similarly for maps, is also a section of the codomain functor, and $T$-linear objects and maps are those for which $T$ behaves like $(-)^2$.

A high-brow way to say this is the following:

\begin{lemma}
\label[lemma]{lemma:cat-of-lin-objs}
    $T$ and $(-)^2$ form parallel 1-cells in the 2-category $\Cat/\cat B$, and the inclusion category of $T$-linear objects and maps in $\cat B$ is the universal fibred functor which makes the two composites naturally isomorphic:
    \begin{equation}
    \label{eq:lin-cat-def}
        \begin{tikzcd}[ampersand replacement=\&, column sep=scriptsize]
            {\Lin(T)} \& {\cat B} \&\& { \cat B^\downarrow} \\
            \&\& {\cat B}
            \arrow["{T}", shift left, from=1-2, to=1-4]
            \arrow["{(-)^2}"', shift right, from=1-2, to=1-4]
            \arrow[Rightarrow, no head, from=1-2, to=2-3]
            \arrow["\cod", from=1-4, to=2-3]
            \arrow["U_\lin", dashed, from=1-1, to=2-3]
            \arrow["U_\lin", dashed, from=1-1, to=1-2]
        \end{tikzcd}
    \end{equation}
\end{lemma}
\begin{proof}
    The universal property of $U_\lin$ is that of an \emph{isoinserter}~\cite[p.~308]{kelly_elementary_1989}.
    An object of the isoinserter $\Lin(T)$ is an object $B$ of $\cat B$ together with an isomorphism $p_B \cong \pi_1$ which is $\cod$-vertical, hence exactly a trivialization of $p_B$.
    Likewise, a map in $\Lin(T)$ is a map that commutes with the trivializations:
    \begin{equation}
        \begin{tikzcd}[ampersand replacement=\&,sep=scriptsize]
            TB \&\&\& {TB'} \\
            \& B \& {B'} \\
            {B \times B} \&\&\& {B' \times B'}
            \arrow["Tf", from=1-1, to=1-4]
            \arrow[" t ", from=1-4, to=3-4]
            \arrow[" t "', from=1-1, to=3-1]
            \arrow["{f \times f}"', from=3-1, to=3-4]
            \arrow["f"{pos=0.6}, from=2-2, to=2-3]
            \arrow["{p_{B'}}"{description}, from=1-4, to=2-3]
            \arrow["{\pi_1}"{description}, from=3-4, to=2-3]
            \arrow["{\pi_1}"{description}, from=3-1, to=2-2]
            \arrow["{p_B}"{description}, from=1-1, to=2-2]
        \end{tikzcd}
    \end{equation}
    The commutativity of the above hinges on the commutativity of the outer square, since all the other subdiagrams commute for various reasons.
    But the outer diagram commutes iff $f$ is $T$-linear~(cf.~\eqref{eq:lin-map}).
\end{proof}

Notice there can be many trivializations for the same object.
For instance, in the case $\cat B = \Smooth$, $(T,p) = (T, \pi)$,  the tangent space of a cartesian space such as $\R^n$ is isomorphic to $\pi_1:{\R^n \times \R^n \to \R^n}$ but clearly any linear diffeomorphism of the latter can be used to get a different trivialization.

\begin{lemma}
\label[lemma]{lemma:lin-of-maps}
    Suppose $f:B \to B'$ is a map in $\cat B$.
    Then $f$ is $T$-linear iff the following commutes:
    \begin{equation}
    \label{eq:f-comm-with-snd-proj}
        \begin{tikzcd}[ampersand replacement=\&,sep=scriptsize]
            TB \& {TB'} \\
            {B \times B} \& {B' \times B'} \\
            B \& {B'}
            \arrow["Tf", from=1-1, to=1-2]
            \arrow[" t ", from=1-2, to=2-2]
            \arrow[" t "', from=1-1, to=2-1]
            \arrow["{\pi_2}"', from=2-1, to=3-1]
            \arrow["{\pi_2}", from=2-2, to=3-2]
            \arrow["f"', from=3-1, to=3-2]
        \end{tikzcd}
    \end{equation}
\end{lemma}
\begin{proof}
    Both directions are direct applications of the universal property of the cartesian product, once we notice that the apparently missing diagram:
    \begin{equation}
        \begin{tikzcd}[ampersand replacement=\&,sep=small]
            TB \& {TB'} \\
            {B \times B} \& {B' \times B'} \\
            B \& {B'}
            \arrow["Tf", from=1-1, to=1-2]
            \arrow[" t ", from=1-2, to=2-2]
            \arrow[" t "', from=1-1, to=2-1]
            \arrow["{\pi_1}"', from=2-1, to=3-1]
            \arrow["{\pi_1}", from=2-2, to=3-2]
            \arrow["f"', from=3-1, to=3-2]
        \end{tikzcd}
    \end{equation}
    commutes automatically, since the vertical sides are equal to $p_B$, and thus this diagram is the commutative square $T$ maps $f$ to.
\end{proof}

\begin{lemma}
\label[lemma]{lemma:lin-prod}
    If $T$ is product-preserving, then $\Lin(T)$ has products and $U_\lin:\Lin(T) \to \cat B$ is product-creating.
\end{lemma}
\begin{proof}
    Let $B,B':\Lin(T)$, then we have
    \begin{equation}
        T(B \times B') \cong TB \times TB' \cong (B \times B) \times (B' \times B') \cong (B \times B') \times (B \times B'),
    \end{equation}
    where all the isomorphisms are intended to be over $B \times B'$.
    Trivially, $U_\lin:\Lin(T) \to \cat B$ preserves products in this way.
    Product-reflection comes from the fact that product projections are $T$-linear, as can be evinced from the following:
    \begin{equation}
        \begin{tikzcd}[ampersand replacement=\&,row sep=scriptsize]
            {T(B \times B')} \& TB \\
            {TB \times TB'} \& TB \\
            {(B \times B') \times (B \times B')} \& {B \times B}
            \arrow["{T\pi_B}", from=1-1, to=1-2]
            \arrow["{\pi_B \times \pi_B}", from=3-1, to=3-2]
            \arrow["{ t  \times  t }"', from=2-1, to=3-1]
            \arrow["\cong"', from=1-1, to=2-1]
            \arrow[" t ", from=2-2, to=3-2]
            \arrow["\pi_{TB}", from=2-1, to=2-2]
            \arrow[Rightarrow, no head, from=1-2, to=2-2]
        \end{tikzcd}
    \end{equation}
\end{proof}

\subsection{Linearity in FODS}
We can instantiate~\cref{def:linearity} for any first-order cartesian tangent structure $\FOTS$, meaning $\cat B$ is a category with bundles and finite products while $T$ is a product-preserving section of the fibration of additive bundles of $\cat B$.

\begin{definition}
    We say an object in $\cat B$ is \textbf{$T$-linear} iff it is linear in the sense of~\cref{def:linearity} with respect to the section $T \colon \cat B \to \cat B^\downarrow$.
\end{definition}

We keep denoting the resulting category of $T$-linear objects by $\Lin(T)$.

By definition, on $T$-linear objects, $T$ behaves as a first-order CDC (\cref{th:FODS-CDCs}), sending an object $B$ to the projection $\pi_1 \colon B \times B \to B$.
This fact will be fully expounded later on, where we will deal with building an actual CLA category on which a first-order CDC can be defined.
For now, we limit ourselves to observing the following:

\begin{lemma}
\label[lemma]{lemma:FODS-is-CDC-on-lin-objs}
    There is a commutative square of fibrations and sections:
    \begin{equation}
        \begin{tikzcd}[ampersand replacement=\&]
            {S^+(\Lin(T))} \& {S^+(\cat B)} \& {{\cat B^\downbundle}} \\
            {\Lin(T)} \&\& {\cat B}
            \arrow["\cod", from=1-3, to=2-3]
            \arrow["\simple", from=1-1, to=2-1]
            \arrow["{U_\lin}", from=2-1, to=2-3]
            \arrow["{{(-)^2}}", curve={height=-6pt}, dashed, from=2-1, to=1-1]
            \arrow["T", curve={height=-6pt}, dashed, from=2-3, to=1-3]
            \arrow["{{S^+(U_\lin)}}", from=1-1, to=1-2]
            \arrow[from=1-2, to=1-3]
        \end{tikzcd}
    \end{equation}
    where the arrow stands for the inclusion of $S^+(\cat B)$ into $\cat B^{\downbundle+}$.
\end{lemma}
\begin{proof}
    Direct corollary of~\cref{lemma:cat-of-lin-objs} and~\cref{lemma:lin-prod}.
\end{proof}


\subsection{Linearity and additivity}
In this section, we keep fixed the cartesian FODS $\FOTS$ just introduced.
As anticipated, our notion of $T$-linear objects and maps generalises that of differential object in the tangent categories literature. We will show that differential objects in the sense of~\cite[Definition~3.1]{cockett_differential_2016} are basically $T$-linear commutative monoids.
The only difference between the latter and the first is one compatibility axiom that we cannot yet state---it would require `higher-order' technology, which we don't deal with in this work.
We first give a classical definition, which differes from~\cite[Definition~3.1]{cockett_differential_2016} only in the aforementioned extra axiom:

\begin{definition}
\label[definition]{def:first-order-diff-objs}
    A \textbf{first-order $T$-differential object} is an object $B:\cat B$ together with
    \begin{enumerate}
        \item a map $\hat p_B : TB \to B$ making $B \nfrom{p_B} TB \nto{\hat p_B} B$ into a product cone and
        \item a commutative monoid structure $(0_B, +_B)$,
    \end{enumerate}
    such that the following commute (we suppressed subscripts where this does not create ambiguity):
    \begin{equation}
    \label{eq:diff-obj-axioms-tau-add}
        \begin{tikzcd}[ampersand replacement=\&, sep=scriptsize]
            B \& TB \\
            1 \& B
            \arrow["{!}"', from=1-1, to=2-1]
            \arrow["0_B", from=2-1, to=2-2]
            \arrow["{\hat p}", from=1-2, to=2-2]
            \arrow["{0_{TB}}", from=1-1, to=1-2]
        \end{tikzcd}
        \quad
        \begin{tikzcd}[ampersand replacement=\&, sep=scriptsize]
            {TB \oplus_B TB} \& TB \\
            {B \times B} \& B
            \arrow["{\langle \pi_1\hat p, \pi_2\hat p \rangle}"', from=1-1, to=2-1]
            \arrow["{+_B}", from=2-1, to=2-2]
            \arrow["{\hat p}", from=1-2, to=2-2]
            \arrow["{+_{TB}}", from=1-1, to=1-2]
        \end{tikzcd}
    \end{equation}
    \begin{equation}
    \label{eq:diff-obj-axioms-cmon-lin}
        \begin{tikzcd}[ampersand replacement=\&,sep=scriptsize]
            T1 \& TB \\
            1 \& B
            \arrow["\cong"', from=1-1, to=2-1]
            \arrow["{0_B}", from=2-1, to=2-2]
            \arrow["{\hat p}", from=1-2, to=2-2]
            \arrow["{T(0_B)}", from=1-1, to=1-2]
        \end{tikzcd}
        \quad
        \begin{tikzcd}[ampersand replacement=\&,sep=scriptsize]
            {T(B \times B)} \& TB \\
            {TB \times TB} \\
            {B \times B} \& B
            \arrow["{T(+_B)}", from=1-1, to=1-2]
            \arrow["{+_B}", from=3-1, to=3-2]
            \arrow["{\hat p}", from=1-2, to=3-2]
            \arrow["\cong"', from=1-1, to=2-1]
            \arrow["{\hat p \times \hat p}"', from=2-1, to=3-1]
        \end{tikzcd}
    \end{equation}
\end{definition}


In the above, $+_{TB}$ and $0_{TB}$ refer to the commutative monoid structure of $p_B:TB \to B$ in the fiber of $\cat B^{\downbundle+}$ over $B$, while $T1 \cong 1$ and $T(B \times B) \cong TB\times TB$ come from the product-preserving assumption on $T$.

\begin{remark}
    The first diagram of~\eqref{eq:diff-obj-axioms-cmon-lin} is actually redundant, as shown in~\cite[Lemma~4.9]{cockett_differential_2014}.
    We include it for clarity.
\end{remark}

Let's try to understand what's going on with~\cref{def:first-order-diff-objs}, from our perspective.
Define $T^+$ as the following change of base for $T$:
\begin{equation}
    \begin{tikzcd}[ampersand replacement=\&, column sep=scriptsize]
        {U^*\cat B^{\downbundle+}} \& {\cat B^{\downbundle+}} \\
        {\CMon(\cat B)} \& {\cat B}
        \arrow[""{name=0, anchor=center, inner sep=0}, "U"', from=2-1, to=2-2]
        \arrow["\cod"', from=1-2, to=2-2]
        \arrow[from=1-1, to=1-2]
        \arrow[from=1-1, to=2-1]
        \arrow["T"', curve={height=6pt}, dashed, from=2-2, to=1-2]
        \arrow["{T^+}", curve={height=-6pt}, dashed, from=2-1, to=1-1]
        \arrow["\lrcorner"{anchor=center, pos=0.125}, draw=none, from=1-1, to=0]
    \end{tikzcd}
\end{equation}
The section $T^+$ doesn't differ much from $T$: it operates by first forgetting the commutative monoid structure of its argument and then proceeding as $T$.
Thus a $T^+$-linear object is a commutative monoid $(0_B, +_B)$ with an additive trivialization $ t $.

In particular, $ t $ identifies the additive bundle $(p_B, 0_{TB}, +_{TB})$ (a commutative monoid in $\cat B^{\bundle}_B$) with $(\pi_1, 0_{B \times B}, +_{B \times B})$, where the latter has been obtained by lifting the commutative monoid structure of $B$ as follows:
\begin{equation}
    0_{B \times B}\ \dfrac{b:B \vdash 1 : 1}{b:B \vdash 0_B :B},
    \quad
    +_{B \times B}\ \dfrac{b:B \vdash (b_1, b_2) : B \times B}{b:B \vdash b_1 +_B b_2 :B}.
\end{equation}
Without resorting to elements, we can characterise $0_{B \times B}$ and $+_{B \times B}$ by saying the following commute:
\begin{equation}
\label{eq:tplus-lin-obj}
    \begin{tikzcd}[ampersand replacement=\&,sep=scriptsize]
        B \& {B \times B} \\
        1 \& B
        \arrow["{0_B}", from=2-1, to=2-2]
        \arrow["{\pi_2}", from=1-2, to=2-2]
        \arrow["{!_B}"', from=1-1, to=2-1]
        \arrow["{0_{B \times B}}", from=1-1, to=1-2]
    \end{tikzcd}
    \quad
    \begin{tikzcd}[ampersand replacement=\&,sep=scriptsize]
        {(B \times B) \oplus_B (B \times B)} \& {B \times B} \\
        {B \times B} \& B
        \arrow["{0_B}", from=2-1, to=2-2]
        \arrow["{\pi_2}", from=1-2, to=2-2]
        \arrow["{+_{B \times B}}", from=1-1, to=1-2]
        \arrow["{\langle \pi_1 \pi_2, \pi_2\pi_2 \rangle}"', from=1-1, to=2-1]
    \end{tikzcd}
\end{equation}
The reader might recognise in these the two diagrams of~\eqref{eq:diff-obj-axioms-tau-add}: in effect, those are saying that $p_B$ is equipped with an additive structure that doesn't use its first argument (i.e.~it's constant over $B$).

What about the second set of axioms?
Without squinting too hard (aided by~\cref{lemma:lin-of-maps}), we see it asserts that the commutative monoid structure of $B$ is $T$-linear.\footnote{The fact $0_B$ is $T$-linear follows from the first axiom, as already observed in~\cite[Lemma~4.9]{cockett_differential_2014}.}
However, this is not neccessarily the case for the objects of $\Lin(T^+)$!

On the other hand, consider $\CMon(\Lin(T))$, i.e.~$T$-linear commutative monoids.
An object here is a $T$-linear object $(B,  t )$ equipped with a commutative monoid structure $(0_B, +_{B \times B})$, where these maps are themselves $T$-linear.
Thus, by definition, these satisfy~\eqref{eq:diff-obj-axioms-cmon-lin}.
However, the trivialization $ t $ isn't additive.
In fact, the additive bundle structure on $\pi_1 : B \times B \to B$ isn't even defined at this stage!

So we are facing a chicken-and-egg problem here: to define additivity of $ t $, we need to have a commutative monoid structure on $B$, but to say this structure is linear, we need to have a trivialization.

To solve this problem and characterise first-order $T$-differential object, we can thus resort to taking a (strict 2-)limit:
\begin{equation}
\label{eq:diff-fst-pb}
    \begin{tikzcd}[ampersand replacement=\&,sep=scriptsize]
        \& {\Diff(T)} \\
        {\Lin(T^+)} \&\& {\CMon(\Lin(T))} \\
        {\Lin(T)} \&\& {\CMon(\cat B)} \\
        \& {\cat B}
        \arrow["{U_{\lin}}", from=2-3, to=3-3]
        \arrow["{U_\lin}"{description, pos=0.73}, curve={height=-6pt}, from=2-1, to=3-3]
        \arrow["{U_\add}"{description, pos=0.7}, curve={height=6pt}, from=2-3, to=3-1]
        \arrow["{U_\add}"', from=2-1, to=3-1]
        \arrow[dashed, from=1-2, to=2-1]
        \arrow[dashed, from=1-2, to=2-3]
        \arrow["{U_\lin}"', from=3-1, to=4-2]
        \arrow["{U_\add}", from=3-3, to=4-2]
    \end{tikzcd}
\end{equation}
Here the functors $U_\lin$ and $U_\add$ forget linear and additive structure respectively.

In this way, an object $B$ of the suggestively-named $\Diff(T)$ is equipped with a commutative monoid structure $(0_B,+_B)$ and a trivialization $ t $ such that the second is additive with respect to the additive bundle structure on $\pi_1:B \times B \to B$ induced by the first, and the first is $T$-linear with respect to latter trivialization.

To justify the name given to $\Diff(T)$, we prove the following:

\begin{therm}
\label[theorem]{th:cmon-lin-obj-is-diff-obj}
    The following are equivalent:
    \begin{enumerate}
        \item objects of $\Diff(T)$,
        \item first-order $T$-differential objects.
    \end{enumerate}
\end{therm}

We'll deal with maps just after the proof.

\begin{proof}
    Assume $(B, 0_B, +_B,  t ): \Diff(T)$.
    The commutative monoid structure on $B$ is already given, while the map $\hat p_B$ can be defined as $TB \nlongto{ t } B \times B \nlongto{\pi_2} B$.
    It is immediate to see $B \nlongfrom{p_B} TB \nlongto{\hat p_B} B$ is a product cone, since $ t $ gives an isomorphism with $B \nlongfrom{\pi_1} B \times B \nlongto{\pi_2} B$.

    Coming to the axioms listed in~\cref{def:first-order-diff-objs}, as observed above the first two are satisfied since the commutative monoid structure of $\pi_1:B \times B \to B$ satisfies~\eqref{eq:tplus-lin-obj} and $ t $, being an additive bundle isomorphism, allows to conclude the same for $p_B$.

    Conversely, assume $B: \cat B$ is a first-order $T$-differential object.
    First, from the fact $B \nfrom{p_B} TB \nto{\hat p_B} B$ is a product cone, we get a trivialization of $TB$ by setting $ t  := \langle p_B, \hat p_B \rangle$.
    We can show $ t $ is additive by using~\eqref{eq:diff-obj-axioms-tau-add}.
    We exemplify the argument by proving that $ t $ commutes with $+$, the proof for $0$ is analogous.

    Indeed, consider (were we suppressed subscripts for clarity):
    \begin{equation}
        \begin{tikzcd}[ampersand replacement=\&,sep=scriptsize]
            {TB \oplus_B TB} \& TB \\
            {(B \times B) \oplus_B (B \times B)} \\
            {B \times (B \times B)} \& {B \times B}
            \arrow[Rightarrow, no head, from=3-1, to=2-1]
            \arrow["{\langle p, \hat p \rangle \oplus_B \langle p, \hat p \rangle}"', from=1-1, to=2-1]
            \arrow["{\langle p, \hat p \rangle}", from=1-2, to=3-2]
            \arrow["{+}", from=1-1, to=1-2]
            \arrow["{+}"', from=3-1, to=3-2]
        \end{tikzcd}
    \end{equation}
    By universal property of the products, commutativity of the diagram above is equivalent to the commutativity of the two diagrams below:
    \begin{equation}
        \begin{tikzcd}[ampersand replacement=\&,sep=scriptsize]
            {TB \oplus_B TB} \& TB \\
            B \& B
            \arrow["p", from=1-2, to=2-2]
            \arrow["{+_{TB}}", from=1-1, to=1-2]
            \arrow[Rightarrow, no head, from=2-1, to=2-2]
            \arrow["{p \oplus_B p}"', from=1-1, to=2-1]
        \end{tikzcd}
        \quad
        \begin{tikzcd}[ampersand replacement=\&,sep=scriptsize]
            {TB \oplus_B TB} \& TB \\
            {B \times B} \& B
            \arrow["{\hat p}", from=1-2, to=2-2]
            \arrow["{+}", from=1-1, to=1-2]
            \arrow["{+_B}", from=2-1, to=2-2]
            \arrow["{\langle \pi_1 \hat p, \pi_2 \hat p \rangle}"', from=1-1, to=2-1]
        \end{tikzcd}
    \end{equation}
    But the first commutes because $+_{TB}$ is a vertical morphism of bundles, and the second by assumption of~\cref{eq:diff-obj-axioms-tau-add}.

    It remains to prove the commutative monoid structure on $B$ is $T$-linear, which is a trivial consequence of~\eqref{eq:diff-obj-axioms-cmon-lin} and~\cref{lemma:lin-of-maps}.
\end{proof}

\cite{cockett_differential_2016} define `differential linear maps' between differential objects.
Their maps turn out to be just $T$-linear maps:

\begin{therm}
\label[theorem]{th:lin-map-is-diff-lin-map}
    The following are equivalent:
    \begin{enumerate}
        \item a map $f:B \to B'$ in $\Lin(T)$,
        \item a map $f:B \to B'$ between objects equipped with a $\hat p_B$ as in~\cref{def:first-order-diff-objs} satisfying:
        \begin{equation}
        \label{eq:diff-lin-map}
            \begin{tikzcd}[ampersand replacement=\&,sep=scriptsize]
                TB \& {TB'} \\
                B \& {B'}
                \arrow["Tf", from=1-1, to=1-2]
                \arrow["{\hat p_B}"', from=1-1, to=2-1]
                \arrow["{\hat p_{B'}}", from=1-2, to=2-2]
                \arrow["f", from=2-1, to=2-2]
            \end{tikzcd}
        \end{equation}
    \end{enumerate}
\end{therm}
\begin{proof}
    Clearly a $T$-linear map satisfies (2) in light of the previous lemma and of the characterization of $T$-linearity in~\cref{lemma:lin-of-maps}.

    Conversely, defining $ t _B := \langle p_B, \hat p_B \rangle$ and $ t _{B'}$ analogously, we see $f$ is a map between $T$-linear objects which is $T$-linear by~\cref{lemma:lin-of-maps} again.
\end{proof}

However, if we seek a notion of morphism of differential objects then linearity alone doesn't cut it---we believe the correct maps between first-order $T$-differential objects should preserve also the additive structure.

\begin{definition}
    The \textbf{category of first-order $T$-differential objects} $\Diff^\fst(T)$ is the category defined in~\eqref{eq:diff-fst-pb}, whose objects are first-order $T$-differential objects and whose maps are both linear and additive.
\end{definition}

To lighten notation, in the following section we write $\Diff(T)$ for this latter category.

\subsection{Cofree CDCs from linear objects}
In the context of tangent categories,~\cite[Theorem~4.11]{cockett_differential_2014} proves $T$-differential objects form a CLA category with CDC structure.

Here we prove an analogous result, generalising to FODS with comprehension and finite products.

The resulting CDC structure won't be on $\CMon(\Lin(T))$, but on a larger category with the same objects but whose maps have been relaxed to not be necessarily additive nor linear.
This is, in fact, the reason why in the literature maps of differential objects are often taken to be non-additive and non-linear.

We can motivate the usage of such a `wonky' category (appearing at the very bottom of the diagram below) by recognizing it has a property not unlike that of an image.
In fact, bijective-on-objects and fully faithful functors form a factorization system in $\Cat$~\cite[]{dupont_proper_2003}, yielding the following factorizations:
\begin{equation}
\label{eq:linadd-fact}
    \begin{tikzcd}[ampersand replacement=\&,sep=small]
        {\Diff(T)} \&\& {\CMon(\cat B)} \&\& {\cat B} \\
        \& {\Diff(T)^\nlin} \&\& {\CMon(\cat B)_\nadd} \\
        \&\& {\Diff(T)^\nlin_\nadd}
        \arrow["{U_\lin}", from=1-1, to=1-3]
        \arrow["\bo"', from=1-1, to=2-2]
        \arrow["\ff"', from=2-2, to=1-3]
        \arrow["{U_\add}", from=1-3, to=1-5]
        \arrow["\bo"', from=1-3, to=2-4]
        \arrow["\ff"', from=2-4, to=1-5]
        \arrow["\bo"', "i", from=2-2, to=3-3]
        \arrow["\ff"', from=3-3, to=2-4]
    \end{tikzcd}
\end{equation}
Thus $\Diff(T)_\nadd^\nlin$ is the full image of the composite $U_\lin U_\add$.
Concretely, $\Diff(T)_\nadd^\nlin$ is constructed by taking the objects from $\Diff(T)$ and defining
\begin{equation}
    \Diff(T)_\nadd^\nlin(A,B) = \cat B(U(A), U(B)).
\end{equation}

This way of constructing $\Diff(T)_\nadd^\nlin$ highlights how such a category is bestowed with CLA structure.
In fact, in light of~\cref{lemma:CLA-charac}, this is exactly what the functor $i$ in~\eqref{eq:linadd-fact} amounts to.

\begin{lemma}
    The functor $i:\Diff(T)^\nlin \to \Diff(T)_\nadd^\nlin$, defined in~\eqref{eq:linadd-fact}, is an equipment of additive maps (as defined in~\cref{lemma:CLA-charac}).
\end{lemma}
\begin{proof}
    Trivially, $i$ is bijective-on-objects.
    Also, it's easy to convice oneself the induced functor $M:\Diff(T)^\nlin \to \CMon(\Diff(T)_\nadd^\nlin)$ is fully faithful.

    To see $i$ is product-preserving, observe that for $B,B':\Diff(T)^\nlin$:
    \begin{equation}
        i(B \nlongfrom{\pi_B} B \oplus B' \nlongto{\pi_{B'}} B')
        = B \nlongfrom{U_\add(\pi_B)} B \times B' \nlongto{U_\add(\pi_{B'})} B'
        = B \nlongfrom{\pi_{B}} B \times B' \nlongto{\pi_{B'}} B'
    \end{equation}
    To see it is also product-reflecting, observe projections are always additive.
\end{proof}

\begin{corollary}
    $\Diff(T)_\nadd^\nlin$ is a CLA category: a first-order $T$-differential object $B$ is equipped with its very commutative monoid structure, and a map is additive iff it preserves it.
\end{corollary}

Hence additivity still lives in $\Diff(T)_\nadd^\nlin$ in the form of a CLA structure.
Instead, linearity is restored by having a first-order CDC structure that $T$ induces on it.

\begin{therm}
\label[theorem]{th:cdc-of-lin-objs}
    Let $\FOTS$ be a first-order cartesian tangent structure.
    Then $T$ induces a first-order CDC structure on the CLA category $\Diff(T)_\nadd^\nlin$.
\end{therm}
\begin{proof}
    The first-order CDC is so defined:
    \begin{equation}
        \begin{tikzcd}[ampersand replacement=\&, sep=scriptsize]
            {S_\CLA(\Diff(T)_\nadd^\nlin)} \& {{A \choose A}} \& {{B \choose B}} \\
            \\
            {\Diff(T)_\nadd^\nlin} \& A \& B
            \arrow["{\simple_\CLA}"', from=1-1, to=3-1]
            \arrow["{D_T}"', curve={height=6pt}, dashed, from=3-1, to=1-1]
            \arrow[""{name=0, anchor=center, inner sep=0}, "f"', from=3-2, to=3-3]
            \arrow[""{name=1, anchor=center, inner sep=0}, "f"', shift right=1, from=1-2, to=1-3]
            \arrow["{ t _A^{-1}Tf t _B}", shift left=1, from=1-2, to=1-3]
            \arrow["{D_T}", shorten <=17pt, shorten >=17pt, maps to, from=0, to=1]
        \end{tikzcd}
    \end{equation}
    Given the definition of $\Diff(T)$, it is evident $D_T$ is well-defined.
    To prove it is a first-order CDC in the sense of~\cref{th:FODS-CDCs}, we just need to verify $D_T$ preserves CLA structure.
    Specifically, we have to show that $D_T$ preserves chosen monoids.
    This follows from observing that differential objects are exactly those that satisfy~\cref{def:first-order-diff-objs}, hence for which $ t $ is additive and such that $T$ preserves their additive structure, as required.
\end{proof}

\section{Conclusions}\label{sec:conclusions}

The fibrational perspective is a natural foundation for categorical differential algebra, given that tangent spaces can be seen as a thickened form of the base space, and differential operators attach to each map in the base space a tangent map in the tangent space. This paper has shown how the diverse and fragmented research area of (first-order) models of differentiation can be unified by a simple concept of this framework: a section of a fibration, where each fibre has commutative monoids which are preserved by reindexing. We contend that this structure, albeit simple, is sophisticated enough to capture something essential about differentiation, namely the interplay between linearity and non-linearity, as formalised through the commutative monoid structure. We believe our approach substantiates the thesis that fibrations should be the central defining concept in models of differential structure, around which a full theory can be built.

In terms of future research, the obvious extension of our work is covering higher order structure such as the symmetry of partial derivatives, the linearity of the derivative in its tangent space, and higher order axioms like the Universality of Vertical Lift (UVL) from tangent categories.  Such higher-order structure is necessary to develop various aspects of differential geometry; for example:
\begin{itemize}
	\item the symmetry of partial derivatives (in a tangent category, this is instantiated by a natural transformation $c: T \circ T \to T \circ T$) and UVL are used to define the Lie bracket of two vector fields (see \cite[Section 3.4]{cockett_differential_2014});
	\item UVL is key to understanding and working with differential bundles in a tangent category (the analog of vector bundles) (see \cite[Section 2.2]{cockett_differential_2018});
	\item the symmetry of partial derivatives is used to define the curvature of a connection in a tangent category (see \cite[Definition 3.17]{connections}).
\end{itemize}

Regarding repeated differentiation, a key reason the approach of tangent categories works is because the derivative lives in the same category as the underlying function, and so the derivative operator can be re-applied to a derivative. We believe this can be integrated within our framework via the fibrational concept of comprehension. Roughly speaking, comprehension allows one to create a tangent functor from a section, thus mimicking the tangent functor approach in our section-based approach. Concretely, taking the second derivative will cache out as building a new fibration on top of our old fibration. A section of this new fibration attaches derivatives to maps that are themselves derivatives. If we use the term {\em higher order differential structures} to represent such models, then this means higher order differential structures will consist of fibred first-order differential structures. This once more highlights the centrality of fibrations in modelling the way tangent spaces are indexed over base spaces. The full development of such a perspective is left for a follow up work.


\bibliographystyle{msclike}
\bibliography{refs.bib}

\begin{thebibliography}{}

\bibitem[Abbott et~al., 2005]{abbott_containers_2005}
{\bf Abbott, M.}, {\bf Altenkirch, T.}, \textbf{and} {\bf Ghani, N.} 2005.
\newblock Containers: {Constructing} strictly positive types.
\newblock {\em Theoretical Computer Science}, 342(1):3--27.

\bibitem[Alvarez{-}Picallo et~al., 2023]{Alvarez-Picallo23}
{\bf Alvarez{-}Picallo, M.}, {\bf Ghica, D.~R.}, {\bf Sprunger, D.},
  \textbf{and} {\bf Zanasi, F.} 2023.
\newblock Functorial string diagrams for reverse-mode automatic
  differentiation.
\newblock In {\em {CSL}}, volume 252 of {\em LIPIcs}, pp. 6:1--6:20. Schloss
  Dagstuhl - Leibniz-Zentrum f{\"{u}}r Informatik.

\bibitem[Bancilhon and Spyratos, 1981]{Bancilhon-lensdatabases}
{\bf Bancilhon, F.} \textbf{and} {\bf Spyratos, N.} 1981.
\newblock Update semantics of relational views.
\newblock {\em ACM Trans. Database Syst.}, 6(4).

\bibitem[Blute et~al., 2009]{blute_cartesian_2009}
{\bf Blute, R.}, {\bf Cockett, R.}, \textbf{and} {\bf Seely, R.} 2009.
\newblock Cartesian differential categories.
\newblock {\em Theory and Applications of Categories}, 22:622--672.

\bibitem[Blute et~al., 2006]{blute_differential_2006}
{\bf Blute, R.~F.}, {\bf Cockett, J. R.~B.}, \textbf{and} {\bf Seely, R. a.~G.}
  2006.
\newblock Differential categories.
\newblock {\em Mathematical Structures in Computer Science}, 16(6):1049--1083.
\newblock Publisher: Cambridge University Press.

\bibitem[Cockett and Cruttwell, 2014]{cockett_differential_2014}
{\bf Cockett, J. R.~B.} \textbf{and} {\bf Cruttwell, G.~S.} 2014.
\newblock Differential structure, tangent structure, and {SDG}.
\newblock {\em Applied Categorical Structures}, 22(2):331--417.
\newblock Publisher: Springer.

\bibitem[Cockett and Cruttwell, 2016]{cockett_differential_2016}
{\bf Cockett, J. R.~B.} \textbf{and} {\bf Cruttwell, G.~S.} 2016.
\newblock Differential bundles and fibrations for tangent categories.
\newblock {\em arXiv preprint arXiv:1606.08379}.

\bibitem[Cockett and Cruttwell, 2017]{connections}
{\bf Cockett, J. R.~B.} \textbf{and} {\bf Cruttwell, G. S.~H.} 2017.
\newblock Connections in tangent categories.
\newblock {\em Theory and Applications of Categories}, 32(26):835--888.

\bibitem[Cockett and Cruttwell, 2018]{cockett_differential_2018}
{\bf Cockett, J. R.~B.} \textbf{and} {\bf Cruttwell, G. S.~H.} 2018.
\newblock Differential bundles and fibrations for tangent categories.
\newblock {\em Cahiers de Topologie et G{\'e}om{\'e}trie Diff{\'e}rentielle
  Cat{\'e}goriques}, LIX:10--92.

\bibitem[Cockett et~al., 2020]{cockett_reverse_2020}
{\bf Cockett, R.}, {\bf Cruttwell, G.}, {\bf Gallagher, J.}, {\bf Lemay,
  J.-S.~P.}, {\bf MacAdam, B.}, {\bf Plotkin, G.}, \textbf{and} {\bf Pronk, D.}
  2020.
\newblock Reverse {Derivative} {Categories}.
\newblock In {\em 28th {EACSL} {Annual} {Conference} on {Computer} {Science}
  {Logic}}.

\bibitem[Cruttwell et~al., 2022a]{cruttwell_gallagher_lemay_pronk_2022}
{\bf Cruttwell, G.}, {\bf Gallagher, J.}, {\bf Lemay, J.-S.~P.}, \textbf{and}
  {\bf Pronk, D.} 2022a.
\newblock Monoidal reverse differential categories.
\newblock {\em Mathematical Structures in Computer Science},
  32(10):1313–1363.

\bibitem[Cruttwell et~al., 2022b]{cruttwell_monoidal_2022}
{\bf Cruttwell, G.}, {\bf Gallagher, J.}, {\bf Lemay, J.-S.~P.}, \textbf{and}
  {\bf Pronk, D.} 2022b.
\newblock Monoidal {Reverse} {Differential} {Categories}.
\newblock arXiv:2203.12478 [cs, math].

\bibitem[Cruttwell and Lemay, 2024]{reverse_tangent}
{\bf Cruttwell, G.} \textbf{and} {\bf Lemay, J.} 2024.
\newblock Reverse tangent categories.
\newblock {\em CSL 2024}, 288(21):1--21.

\bibitem[Cruttwell, 2017]{cruttwell_cartesian_2017}
{\bf Cruttwell, G.~S.} 2017.
\newblock Cartesian differential categories revisited.
\newblock {\em Mathematical structures in computer science}, 27(1):70--91.
\newblock Publisher: Cambridge University Press.

\bibitem[Cruttwell et~al., 2021]{Crutwell21}
{\bf Cruttwell, G. S.~H.}, {\bf Gavranovic, B.}, {\bf Ghani, N.}, {\bf Wilson,
  P.~W.}, \textbf{and} {\bf Zanasi, F.} 2021.
\newblock Categorical foundations of gradient-based learning.
\newblock {\em CoRR}, abs/2103.01931.

\bibitem[Cruttwell et~al., 2022c]{cruttwell_categorical_2022}
{\bf Cruttwell, G. S.~H.}, {\bf Gavranović, B.}, {\bf Ghani, N.}, {\bf Wilson,
  P.}, \textbf{and} {\bf Zanasi, F.} 2022c.
\newblock Categorical {Foundations} of {Gradient}-{Based} {Learning}.
\newblock In {\bf Sergey, I.}, editor, {\em Programming {Languages} and
  {Systems}}, Lecture {Notes} in {Computer} {Science}, pp. 1--28, Cham.
  Springer International Publishing.

\bibitem[Cruttwell and Lemay, 2023]{cruttwell_differential_2023}
{\bf Cruttwell, G. S.~H.} \textbf{and} {\bf Lemay, J.-S.~P.} 2023.
\newblock Differential {Bundles} in {Commutative} {Algebra} and {Algebraic}
  {Geometry}.
\newblock arXiv:2301.05542 [math].

\bibitem[Dupont and Vitale, 2003]{dupont_proper_2003}
{\bf Dupont, M.} \textbf{and} {\bf Vitale, E.~M.} 2003.
\newblock Proper factorization systems in 2-categories.
\newblock {\em Journal of Pure and Applied Algebra}, 179(1):65--86.

\bibitem[Ehrhard, 2018]{Ehrhard18}
{\bf Ehrhard, T.} 2018.
\newblock An introduction to differential linear logic: proof-nets, models and
  antiderivatives.
\newblock {\em Math. Struct. Comput. Sci.}, 28(7):995--1060.

\bibitem[Ehrhard and Regnier, 2003]{EhrhardR03}
{\bf Ehrhard, T.} \textbf{and} {\bf Regnier, L.} 2003.
\newblock The differential lambda-calculus.
\newblock {\em Theor. Comput. Sci.}, 309(1-3):1--41.

\bibitem[Fong and Spivak, 2020]{fong_supplying_2020}
{\bf Fong, B.} \textbf{and} {\bf Spivak, D.~I.} 2020.
\newblock Supplying bells and whistles in symmetric monoidal categories.
\newblock arXiv:1908.02633 [math].

\bibitem[Fox, 1976]{fox_coalgebras_1976}
{\bf Fox, T.} 1976.
\newblock Coalgebras and cartesian categories.
\newblock {\em Communications in Algebra}, 4(7):665--667.

\bibitem[Gambino and Kock, 2013]{gambino_polynomial_2013}
{\bf Gambino, N.} \textbf{and} {\bf Kock, J.} 2013.
\newblock Polynomial functors and polynomial monads.
\newblock In {\em Mathematical {Proceedings} of the {Cambridge} {Philosophical}
  {Society}}, volume 154, pp. 153--192. Cambridge University Press.
\newblock Issue: 1.

\bibitem[Griewank, 2012]{griewank2012invented}
{\bf Griewank, A.} 2012.
\newblock Who invented the reverse mode of differentiation.
\newblock {\em Documenta Mathematica, Extra Volume ISMP}, 389400.

\bibitem[Hedges, 2017]{Hedges17}
{\bf Hedges, J.} 2017.
\newblock Coherence for lenses and open games.
\newblock {\em CoRR}, abs/1704.02230.

\bibitem[Hermida, 1999]{hermida_properties_1999}
{\bf Hermida, C.} 1999.
\newblock Some properties of {Fib} as a fibred 2-category.
\newblock {\em Journal of Pure and Applied Algebra}, 134(1):83--109.

\bibitem[Jacobs, 1999]{jacobs_categorical_1999}
{\bf Jacobs, B.} 1999.
\newblock {\em Categorical logic and type theory}, volume 141 of {\em Studies
  in {Logic} and the {Foundations} of {Mathematics}}.
\newblock Elsevier, 1st edition.

\bibitem[Kelly, 1989]{kelly_elementary_1989}
{\bf Kelly, G.~M.} 1989.
\newblock Elementary observations on 2-categorical limits.
\newblock {\em Bulletin of the Australian Mathematical Society},
  39(2):301--317.
\newblock Publisher: Cambridge University Press.

\bibitem[Kmett, 2019]{lenses-haskell}
{\bf Kmett, E.~A.} 2019.
\newblock lens: Lenses, folds and traversals.
\newblock {H}askell library.

\bibitem[Rosický, 1984]{Rosicky84}
{\bf Rosický, J.} 1984.
\newblock Abstract tangent functors.
\newblock {\em Diagrammes}, 12:JR1--JR11.

\bibitem[Spivak, 2019]{spivak_generalized_2019}
{\bf Spivak, D.~I.} 2019.
\newblock Generalized {Lens} {Categories} via functors \{\${f}:
  {\textbackslash}mathcal {c}{\textasciicircum}\{{\textbackslash}rm op\}
  {\textbackslash}to {\textbackslash}mathsf\{{cat}\}\$\}.

\bibitem[Spivak, 2020]{spivak_poly_2020}
{\bf Spivak, D.~I.} 2020.
\newblock Poly: {An} abundant categorical setting for mode-dependent dynamics.
\newblock arXiv:2005.01894 [math].

\bibitem[Street, 1974]{Street-fib2cat}
{\bf Street, R.} 1974.
\newblock Fibrations and yoneda's lemma in a 2-category.
\newblock In {\bf Kelly, G.~M.}, editor, {\em Category Seminar}, pp. 104--133,
  Berlin, Heidelberg. Springer Berlin Heidelberg.

\bibitem[Streicher, 2023]{streicher_fibered_2023}
{\bf Streicher, T.} 2023.
\newblock Fibered {Categories} a la {Jean} {Benabou}.
\newblock arXiv:1801.02927 [math].

\bibitem[Taylor, 1999]{taylor_practical_1999}
{\bf Taylor, P.} 1999.
\newblock {\em Practical {Foundations} of {Mathematics}}.
\newblock Number v. 59 in Cambridge {Studies} in {Advanced} {Mathematics}.
  Cambridge University Press.

\end{thebibliography}

\end{document}